\documentclass{amsart}

\usepackage{amssymb}
\usepackage{amsmath}
\usepackage{mathtools}
\usepackage{graphicx}
\usepackage{bbm}
\usepackage{array}
\usepackage{eepic}
\usepackage{multicol}
\usepackage{color}
\usepackage[linktocpage=true]{hyperref}
\usepackage[cmtip,all]{xy}
\usepackage{enumitem}
\usepackage{leftidx}
\usepackage[dvipsnames]{xcolor}

\usepackage{parskip}

\makeatletter
\def\thm@space@setup{%
  \thm@preskip=\parskip \thm@postskip=0pt
}
\makeatother

\allowdisplaybreaks

\numberwithin{equation}{section}

\theoremstyle{plain}
\newtheorem{theorem}{Theorem}[section]
\newtheorem{lemma}[theorem]{Lemma}
\newtheorem{proposition}[theorem]{Proposition}

\theoremstyle{definition}
\newtheorem{definition}{Definition}[section]
\newtheorem{problem}[definition]{Problem}

\newtheorem{example}[definition]{Example}
\newtheorem{remark}[definition]{Remark}

\renewenvironment{proof}{{\bf \noindent Proof.}}{\qed}

\newcommand{\equivalent}{\Longleftrightarrow}

\newcommand{\TT}{\mathbb{T}}
\newcommand{\ZZ}{\mathbb{Z}}
\newcommand{\RR}{\mathbb{R}}
\newcommand{\CC}{\mathbb{C}}

\newcommand{\VV}{\mathbb{V}}

\DeclareMathOperator*\supp{\mathrm{supp}}

\DeclareMathOperator*\wstspan{\overline{\mathrm{span}^{\wast}}}

\def\1{\mathbf{1}}
\def\Id{\mathrm{Id}}
\def\H{{H}}

\def\M{\mathcal{M}}
\def\A{\mathcal{A}}
\def\N{\mathcal{N}}

\def\Prj{\mathcal{P}}
\def\B{\mathcal{B}}
\def\K{\mathcal{K}}
\def\Ball{\mathrm{Ball}}
\def\CB{\mathcal{C B}}

\def\Schur{\mathfrak{M}}

\def\L{\mathcal{L}}
\def\wast{{{\mathrm{w}}^\ast}}
\def\op{\mathrm{op}}
\def\cb{\mathrm{cb}}
\def\Aut{\mathrm{Aut}}

\def\row{\mathrm{r}}

\def\V{\mathcal{V}}

\def\R{\mathcal{R}}

\def\algtensor{\otimes_{\mathrm{alg}}}
\def\mintensor{\otimes_{\min}}

\def\weaktensor{ \, \overline{\otimes} \, }
\def\Fubtensor{\otimes_{\mathcal{F}}}
\def\osprojtensor{\, \widehat{\otimes} \,}

\def\Haag{\otimes_{h}}
\def\eHaag{\otimes_{e h}}
\def\nHaag{\otimes_{\sigma h}}

\addtocounter{tocdepth}{-1}

\begin{document}

\title[On Quantum Relations]{
  A Few Observations on Weaver's Quantum Relations
}

\author[Gonz\'alez-P\'erez]{
  Adri\'an M. Gonz\'alez-P\'erez\\
  \MakeLowercase{
    \texttt{adrian.gonzalez@icmat.es}
  }
}
\thanks{
  The author has been partially supported by the FPI scholarship
  \texttt{BES-2011-044193} and by the Severo Ochoa Excellence Programme
  \texttt{SEV-2011-0087}
}


\maketitle

\begin{abstract}
  In \cite{Wea2012}, Weaver introduced the concept of quantum relation $\R$
  over a von Neumann algebra $\M$. When $\M$ is either finite dimensional
  or discrete and abelian, $\R$ is given by an orthogonal projection in
  $\M \weaktensor \M_\op$. Here, we generalize such result to general
  von Neumann algebras, proving that quantum relations are in bijective
  correspondence with weak-$\ast$ closed left ideals inside
  $\M \eHaag \M$, where $\eHaag$ is the extended Haagerup tensor product.
  The correspondence between the two is given by identifying
  $\M \eHaag \M$ with $\M'$-bimodular operators and proving a double
  annihilator relation
  
  Given an action of a group/quantum group on $\M$ we give a definition for
  invariant quantum relations and prove that in the case of group von Neumann
  algebras $\L G$, invariant quantum relations are left ideals in the measure
  algebra $M G$. At the end we explore possible applications to noncommutative
  harmonic analysis, in particular noncommutative Gaussian bounds.
\end{abstract}


\section{{\bf Prerequisites \label{S1}}}

\subsection{Weaver's Quantum Relations} 

In {\cite{Wea2012,KuWea2012}} Kuperberg and Weaver introduced the
concept of a quantum relation over a von Neumann algebra $\M \subset \B(\H)$.
They defined a quantum relation to be a weak-$\ast$ closed operator bimodule
over $\M'$, i.e.: a linear weak-$\ast$ closed subset $\V \subset \B(\H)$
satisfying that $\M' \, \V \, \M' \subset \V$. It is easy to see that such notion
doesn't depend on the representation $\M \subset \B(\H)$.

In the case $\M = \ell_\infty(X) \subset \B(\ell_2 X)$ acting by multiplication
operators we have that $\M' = \M$. Identifying $\B(\ell_2 X)$ with matrices
indexed by $X \times X$, gives that $\V \subset \B(\ell_2 X)$ is a quantum
relation whenever
\begin{equation}
  \label{S1.eq.lrmult}
  [a_{x \, y}]_{{x , y} \in X} \in \V \implies [b_{x} \, a_{x \, y} c_{y}]_{{x , y} \in X} \in \V,
\end{equation}
for every $(b_x)_{x \in X}$, $(c_x)_{x \in X}$. This in turns easily implies,
see \cite[Proposition 1.3]{Wea2012}, that there is a unique subset
$R \subset X \times X$ such that
\[
  \V_R = \{ [a_{x \, y}]_{x , y} : (x,y) \not\in R \implies a_{x \, y} = 0 \}.
\]
and reciprocally every such subset $R \subset X \times X$ have associated the operator bimodule
of all matrices supported on $R$. When $\M = L_\infty(X) \subset \B(L_2 X)$
is abelian but not atomic we do not have a bijective correspondence between
$\M$ bimodules and measurable subsets of $X \times X$. In that case the
natural object to substitute the (discrete) relations $R \subset X \times X$
will be the, so called, measurable relations, i.e. weak-$\ast$ open subsets
$\R \subset \Prj(\M) \times \Prj(\M)$ satisfying that
\[
  \big( \bigvee_{\alpha} P_\alpha, \bigvee_{\beta} Q_\beta \big) \in \R
  \equivalent
  \exists \alpha_0, \beta_0 \, (P_{\alpha_0}, Q_{\beta_0}) \in \R.
\]
The measurable relation associated with a quantum relation
$\V \subset \B(L_2(X))$ is given by
\begin{equation}
  \label{S1.eq.RfromV}
  \R_\V = \{ (P,Q) \in \Prj(\M) \times \Prj(\M) : P \, \V \, Q \neq \{0\} \}.
\end{equation}
Notice that in the abelian discrete case we have that $\R$ is just the
set of projections $(\chi_A,\chi_B)$ such that there are $x \in A$ and
$y \in B$ with $(x,y) \in R$. Reciprocally, given any measurable relation
$\R$ we can associate a quantum relation over $\M$ given by
\begin{equation}
  \label{S1.eq.VfromR}
  \V_\R = \{ T \in \B(L_2 X) : P \, T \, Q = 0, \, \forall \, (P,Q) \not\in \R \}.
\end{equation}
It is proved in \cite{Wea2012} that the map $\R \mapsto \V_\R$ is injective.
Unfortunately it is not surjective in general. This has to do with the fact
that all the operator bimodules $\V$ arising like in \ref{S1.eq.VfromR} are
not just weak-$\ast$ closed but \emph{operator reflexive}, see
{\cite{Er1986, Lars1982}} and in particular closed in the weak operator
topology, or WOT in short. The way to fix that is to observe that if
$\V \subset \B(\H)$ is any weak-$\ast$ closed linear subspace
$\1 \otimes \V \subset \B(\ell_2 \otimes_2 H)$ is operator
reflexive. Since $\1 \otimes \V$ is a $\CC \1 \otimes \M'$-bimodule and
$(\CC \1 \otimes \M')' = \B(\ell_2) \weaktensor \M$ we have that
$\1 \otimes \V$ is a quantum relation over the amplified algebra
$\B(\ell_2) \weaktensor \M$. This suggests that the right definition
for quantum relations as pairs of related projections is given by amplified
projections in $\B(\ell_2) \weaktensor \M$. The next definition captures this
intuition.

\begin{definition}{{\bf (\emph{\cite[Definition 2.24]{Wea2012}})}}
  \label{S1.def.IQR}
  We will say that $\R \subset \Prj(\M \weaktensor \B(\ell_2)) \times
  \Prj(\M \weaktensor \B(\ell_2))$ is an \emph{intrinsic quantum relation}
  iff
  \begin{enumerate}[label={(\roman*)}, ref=(\roman*)]
    \item \label{S1.def.IQR.1} $\R$ is weak-$\ast$ open.
    \item \label{S1.def.IQR.2} $(0,0) \not\in \R$.
    \item \label{S1.def.IQR.3} If
    $(P_\alpha)_{\alpha \in A}$ and $(Q_\beta)_{\beta \in B}$ are sets
    of families of projections in $\Prj(\M \weaktensor \B(\ell_2))$ then
    \[
      \Big( \bigvee_{\alpha \in A} P_\alpha,
            \bigvee_{\beta \in B} Q_\beta
      \Big) \in \R
      \equivalent
      \exists \alpha_0 \in A, \beta_0 \in B \mbox{ such that }
      (P_{\alpha_0}, Q_{\beta_0}) \in \R.
    \]
    \item \label{S1.def.IQR.4} For every $B \in \1 \otimes \B(\ell_2)$
    we have that
    \[
      ([B P], Q) \in \R \Longleftrightarrow (P,[B^{\ast} Q]) \in \R,
    \]
    where $[A]$ represents the left (or final) projection of the
    operator $A$.
  \end{enumerate}
\end{definition}

Quantum relations over $\M \subset \B(\H)$ and intrinsic quantum relations
(or i.q.r.) over $\M$ are in bijective correspondence and the adaptations
of the maps \ref{S1.eq.RfromV} and \ref{S1.eq.VfromR} are inverse of each
other. Indeed, this correspondence works for every von Neumann algebra
$\M \subset \B(\H)$ not necessarily abelian or discrete, see
\cite[Theorem 2.32]{Wea2012}.

Through this article we are going to employ liberally the language of operator
spaces, see {\cite{Pi2003, EffRu2000Book, BleMer2004Operator}} for more
information. An operator space is a closed linear subset $E \subset \B(\H)$.
Given two operator spaces $E \subset \B(\H_1)$ and $F \subset \B(\H_2)$ we say
that a linear map $\phi : E \to F$ is \emph{completely bounded}, or \emph{c.b.}
in short, iff the matrix amplifications
$\Id \otimes \phi : M_n[E] \subset \B(\ell_2^n \otimes_2 \H_1)
\to M_n[F] \subset \B(\ell_2^n \otimes_2 \H_2)$ are uniformly bounded on $n$.
We are going to denote by $\CB(E,F)$ the space of all completely bounded
(or c.b.) operators with the norm given by
\[
  \| \phi \|_{\cb} = \sup_{n \geq 1} \big\{ \| \Id \otimes \phi : M_n[E] \to M_n[F] \| \big\}.
\]
The category of operator spaces is the collection of all operator spaces with
c.b. maps as morphisms. There is also an intrinsic characterization of operator
spaces as Banach spaces endowed with collections of matrix norms satisfying
the Ruan's axioms. Either an isometric injection $j : E \to \B(\H)$ or a
family of compatible matrix norm will be called an \emph{operator space
structure}, or \emph{o.s.s.} in short.

Let $X$ be a discrete measure space with the counting measure and let us
identify $\B(\ell_2 X)$ with matrices indexed by $X$. Given a matrix
$m = [m_{x \, y}]_{x \, y \in X}$ we define the \emph{Schur multiplier} of
symbol $m$ as the operator $S_m$ given by
\[
  S_m ([a_{x \, y}]) = [ m_{x \, y} \, a_{x \, y} ].
\]
Whenever $S_m$ is completely bounded we will say that $S_m$ is a c.b.
Schur multiplier. We are going to denote by $\Schur(X)
\subset \CB(\B(\ell_2 X))$ the set of all c.b. Schur multipliers and by
$\Schur^\sigma(X)$ the space of all c.b. and normal ones (i.e. weak-$\ast$
continuous for $S_1(\ell_2 X)^\ast = \B(\ell_2 X)$). Assume that $X$ is a
finite set, let $R \subset X \times X$ be a relation and
$\V \subset \B(\ell_2 X)$ be its associated quantum relation. We have that
the ideal $J \subset \Schur^\sigma(X) = \Schur(X)$ given by
\begin{equation}
  J = \{ S \in \Schur(X) : S{|}_\V = 0\}
\end{equation}
contains just the Schur multipliers $S_m$ whose symbol $m$ satisfies that
$m_{x \, y} = 0$ if $(x,y) \in R$. The reciprocal is also true and we have
the following. 

\begin{proposition}
  \label{S1.thm.AbeVfromJ}
  Let $X$ be a finite set and $\ell_\infty(X) \subset \B(\ell_2 X)$ and
  $\V \subset \B(\ell_2 X)$ be as above. Then if $J$ an ideal in
  $\Schur^\sigma(X)$ we have that
  \begin{eqnarray*}
    \V_J & = & \{ T \in \B(\ell_2^n) : S(T) = 0, \forall S \in J \}\\
    J_\V & = & \{ S \in \Schur^\sigma(X) : S {|}_\V = 0 \}
  \end{eqnarray*}
  are bijections between the sets of quantum relations and the set of
  ideals of Schur multipliers. Furthermore, the maps $\V \mapsto J_\V$
  and $J \mapsto \V_J$ are inverse of each other.
\end{proposition}

Such result was generalized to general, not necessarily abelian, finite
dimensional von Neumann algebras $\M \subset \B(\H)$ by Weaver \cite{Wea2012}.
For that end recall that $\Schur^\sigma(X)$ is actually equal to the
algebra of all completely bounded normal operators
$S: \B(\ell_2 X) \to \B(\ell_2 X)$ that are $\ell_\infty (X)$-bimodular. We
are going to denote the the algebras of $\M'$-bimodular c.b. normal operators
on $\B(\H)$ by $\CB^\sigma_{\M' \, \M'}(\B(\H))$. It is trivial to see that in
the case of finite dimensional $\M$ we have a bounded, quasi-isometric and
multiplicative map $\Phi : \M \mintensor \M_\op \to
\CB^\sigma_{\M' \, \M'}(\B(\H))$ given by extension of
\begin{equation}
  \label{S1.eq.Phi}
  \Phi(x \otimes y) = (T \mapsto x T y).
\end{equation}
To see that, let $n = \dim \M$, so that, $\B(H)$ is
quasi-isometric to $\ell_2^n \otimes_2 \ell_2^n$ and $\CB^\sigma(\B(\H))$
is quasi isometric to $\B(\ell_2 \otimes_2 \ell_2)$. If $x,y \in \M'$,
we denote by $T_{x \, y}$ the operator given by $S \mapsto x S y$. It
is clear that $\phi \in \B(\ell_2 \otimes_2 \ell_2)$ is $\M'$-bimodular
iff it belongs to the commutant of $\{T_{x \, y}\}_{x, y \in \M'}$ but
such algebra is isomorphic to $\M \otimes \M_\op$ as we claimed. 
If $\V \subset \B(\H)$ is a quantum relation over $\M$ we have that
$J_\V = \{ s \in \M \otimes \M_\op : \Phi_s {|}_\V = 0 \}$ is a left
ideal and therefore is of the form $J_\V = (\M \otimes \M_\op) \, p_\V$
for some $p_\V \in \Prj(\M \otimes \M_\op)$. Furthermore, we have the
following.

\begin{proposition}{{\bf (\cite[Proposition 2.23]{Wea2012})}}
  If $\M$ is finite dimensional the correspondence
  $\V \mapsto p_\V^{\perp}$ defined as above is an order-preserving
  bijection between quantum relations over $\M$ and projections in
  $\M \otimes \M_\op$.
\end{proposition}

In the case of infinite dimensional von Neumann algebras $\M$ the result above
fails and not every quantum relation can be associated with a projection in
$\M \weaktensor \M_\op$. The reason for that is that although the map
$\Phi : \M \otimes \M_\op \to \CB_{\M' \, \M'}(\B(H))$ is bounded and
multiplicative for every finite dimensional algebra $\M$ it is far from
isometric. Indeed its norm explodes with $n = \dim(\M)$. The problem
can be solved by changing the tensor norm from the spatial tensor norm to the
Haagerup tensor norm of the two von Neumann algebras. With that tool at hand we
will be able to prove a generalization of \ref{S1.thm.AbeVfromJ} for general
algebras in the next section.

\subsection{Module Maps and The Haagerup Tensor Product}

Let $E$, $F$ be two operator spaces. We define the bilinear form $\odot :
M_n[E] \times M_n[F] \to M_n[E \algtensor F]$ by
\[
  [x_{i \, j}] \odot [y_{i \, j}] = \Big[ \sum_{k = 1}^n x_{i \, k} \otimes y_{k \, j} \Big]_{i,j}.
\]
Of course such definition makes perfect sense with matrices of different sizes
$\odot : M_{n,m}[E] \times M_{m,l}[F] \to M_{n,l}[E \algtensor F]$ just by
embedding all matrices inside $M_{\max\{n,m,l\}}$ and restricting. The
\emph{Haagerup tensor norm} for $z \in E \algtensor F$ is defined to be
\begin{eqnarray*}
  \| z \|_{h} & = & \inf \{ \| u \|_{M_{1,n}(E)} \, \| u \|_{M_{1,n}(E)} : z = u \odot v \}\\
              & = & \inf \bigg\{ \Big\| \sum_{k = 1}^n x_k x_k^\ast \Big\|^\frac{1}{2} \Big\| \sum_{k = 1}^n y_k^\ast y_k \Big\|^\frac{1}{2} : z = \sum_{k = 1}^n x_k \otimes y_k \bigg\}
\end{eqnarray*}
The Haagerup tensor product $E \Haag F$ is defined as the completion under that
norm. Similarly $E \Haag F$ can be given an o.s.s by defining:
\[
   \| x \|_{M_n[E \Haag F]}
   =
   \inf \{ \| u \|_{M_{n,k}(E)} \, \| u \|_{M_{k,n}(E)} : z = u \odot v \}.
\]
In the case of two dual operator spaces $E^*$ and $F^*$ the \emph{weak-$\ast$
Haagerup tensor product}, introduced in \cite{BleSmi1992} by Blecher and
Smith, is given by
\[
  E^* \otimes_{\wast h} F^* = (E \Haag F)^\ast.
\]
Since the Haagerup tensor norm is self dual, see \cite{EffRu1991SelfDual},
we have that $E^* \Haag F^*$ embeds inside $E^* \otimes_{\wast h} F^*$
isometrically and is weak-$\ast$ dense. This tensor product is a complemented
subspace of the \emph{normal Haagerup tensor product} $E \nHaag F$ introduced
by Effros and Kishimoto \cite{EffKi1987} and which satisfies that
\[
  (E \Haag F)^{\ast \ast} = (E^{\ast \ast} \nHaag F^{\ast \ast}).
\]
In \cite{EffRu2003} Effros and Ruan introduced the \emph{extended Haagerup
tensor product} generalizing the weak-$\ast$ Haagerup tensor to (potentially)
non-dual operator spaces. Indeed if $x = [x_{i \, j}]_{i \, j}$ is a
matrix whose entries are, possibly infinite, sums of simple tensors, we
say that $x \in M_m(E \eHaag F)$ iff
\[
  \| x \|_{M_m(E \eHaag F)} = \inf \{ \| u \|_{M_{m,I}(E)} \, \| v \|_{M_{m,I}(E)} : x = u \odot v \}
\]
for every possible index set $I$. It can be seen that it is enough to
take $I$ to be the smallest cardinality of a dense set in $\H$ with
$E, F \subset \B(\H)$. Particularly when $E$ and $F$ are separable von
Neumann algebras we can take $I$ numerable. In the case of $E^*$, $F^*$ 
being dual operator spaces, we have that
\begin{equation*}
  \begin{array}{lcl}
    E^* \otimes_{\wast h} F^* & = & E^* \eHaag F^*\\
    E^* \nHaag F^*            & = & (E \eHaag F)^*.
  \end{array}
\end{equation*}
The coarser topology in $E^* \eHaag F^*$ making the pairing with every
element in $\M_\ast \eHaag \M_\ast$ continuous is strictly finer than
the weak-$\ast$ topology given by the predual $\M_\ast \Haag \M_\ast$.
Since $E^* \eHaag F^* \subset E^* \nHaag F^*$ is
$\sigma(E^* \eHaag F^*)$-closed, $E^* \eHaag F^*$, with the
$\sigma(E^* \eHaag F^*)$ topology, is a dual space. Its predual is
obtained by a quotient of $E^* \eHaag F^*$

When $E = \N$, $F = \M$ are von Neumann algebras $\N \eHaag \M$ is a
weak-$\ast$ Banach algebra with a jointly completely bounded multiplication,
see \cite[pp. 126]{EffRu2000Book}, given by extension of
\[
  (x \otimes y) \, (z \otimes t) = x \, z \otimes t \, y.
\]
When $\M = \N$ there is also a natural multiplicative involution
$(x \otimes y)^\dagger = y^\ast \otimes x^\ast$.

Recall that the space of completely bounded $\CB(E,F)$ has a natural o.s.s.
given by the identification $M_n(\CB(E,F)) = \CB(E,M_n(F))$. If $E^*$ and
$F^*$ are dual operator spaces we define
$\CB^\sigma(E^*,F^*) \subset \CB(E^*,F^*)$ to be susbspace of all weak-$\ast$
continuous operators. We have a natural identification 
$\CB^\sigma(E^*,F^*) = \CB(F,E)$. When $E,F \subset \B(\H)$ are bimodules
over a von Neumann algebra $\M \subset \B(H)$ we will denote by
$\CB_{\M \, \M}(E,F)$ and $\CB^\sigma_{\M \, \M}(E,F)$ the subspaces of
completely bounded and bimodular operators. Such subspaces are easily seen to
be norm closed. We will treat mainly the case when $E = F = \B(H)$. 
We have, using that $\K(\H)^{\ast \ast} = \B(\H)$ and that
$\CB(E,F^\ast) = \CB^\sigma(E^{\ast \ast}, F^{\ast})$, see
\cite[(1.28)]{BleMer2004Operator}, that
\begin{equation}
  \label{S1.eq.CBsigma}
  \CB^\sigma(\B(\H)) = \CB(\K(\H),\B(H)).
\end{equation}
The identification is given by restriction to $\K(\H) \subset \B(\H)$ and
by passage to the second dual. The identity \ref{S1.eq.CBsigma} allow us to
give a predual for $\CB^\sigma(\B(\H))$ by
\begin{equation}
  \label{S1.eq.CBPredual}
  \begin{array}{rclr}
    \CB(\K(\H),\B(\H)) & = & \CB(\K(\H), \CC) \Fubtensor \B(\H)   & \mbox{(by \cite[Th. 4.1]{Pi2003})}\\
                       & = & (\K(\H) \osprojtensor S_1(\H))^\ast, &
  \end{array}
\end{equation}
where $\Fubtensor$ is the Fubini tensor product, see \cite{EffKraRu1993},
\cite{EffRu2003} or \cite{EffRu2000Book}
which is isomorphic to the dual of the (operator space) projective tensor
product $\osprojtensor$, see \cite[Chap. 7]{EffRu2000Book}. Similarly the
predual of $\CB(\B(\H))$ is given by $\B(\H) \osprojtensor S_1(\H)$. In
both cases the pairing is given by linear extension of
$\langle T \otimes \xi, \Psi \rangle = \langle \xi, \Psi(T) \rangle$, for
$\Psi \in \CB(\B(\H))$. A subtle point is that the coarser topology in
$\CB^\sigma(\B(\H))$ making the paring with all the elements in
$\B(\H) \osprojtensor S_1(H)$ continuous is, in general, strictly finer
than the weak-$\ast$ topology given by the predual
$\K(\H) \osprojtensor S_1(\H)$. To see that, notice that the following
inclusion holds
\[
  \K(\H) \osprojtensor S_1(\H) \subset \B(\H) \osprojtensor S_1(\H).
\]
Indeed, the inclusion above is just a consequence of the fact that
$\K(\H) \subset \B(\H)$ and the injectivity of the functor
$E \mapsto \M_\ast \osprojtensor E$, where $\M_\ast$ is the predual of any
\emph{hyperfinite} von Neumann algebra, see \cite{Pi1998}. Since
$\sigma(\K(\H) \osprojtensor S_1(\H))$-closed sets are
$\sigma(\B(\H) \osprojtensor S_1(\H))$-closed we have that
$\CB^\sigma(\B(\H)) \subset \CB(\B(\H))$ is
$\sigma(\B(\H) \osprojtensor S_1(\H))$-closed and so the
$\sigma(\B(\H) \osprojtensor S_1(\H))$ topology induces another predual
for $\CB^\sigma(\B(\H))$.
Clearly, the topology of pointwise weak-$\ast$ convergence in
$\CB^\sigma(\B(\H))$ is coarser than the ${\sigma(\B(\H) \osprojtensor} S_1(\H))$
topology. Analogously, the topology of pointwise (in $\K(\H)$) weak-$\ast$
topology is coarser that the $\sigma(\K(\H) \osprojtensor S_1(\H))$ topology.
In both cases the topologies coincide over bounded sets.

The subspace of bimodular operators $\CB^\sigma_{\M \, \M}(\B(\H))$ is closed
in both the $\sigma(\K(\H) \allowbreak \osprojtensor S_1(\H))$ and the
$\sigma(\B(\H) \osprojtensor S_1(\H))$ topologies. Indeed, it is closed in the
$\K(\H)$-pointwise weak-$\ast$ topology which is coarser than both. As a
consequence, using the Hanhn-Banach Theorem, we get that
$\CB^\sigma_{\M \, \M}(\B(\H))$ inherits two natural preduals topologies
\begin{eqnarray*}
  \CB^\sigma_{\M \, \M}(\B(\H)) & = & (\B(\H) \osprojtensor S_1(\H) / K_2 )^\ast,\\
  \CB^\sigma_{\M \, \M}(\B(\H)) & = & (\K(\H) \osprojtensor S_1(\H) / K_1 )^\ast,
\end{eqnarray*}
where $K_1$, $K_2$ are the corresponding preannihilators. Similarly
$\CB_{\M \, \M}(\B(\H))$ is also a dual space with the
${\sigma(\B(\H)} \osprojtensor S_1(\H))$ topology.
The spaces $\CB_{\M \, \M}(\K(\H))$, $\CB^\sigma_{\M \, \M}(\B(\H))$ and
$\CB_{\M \, \M}(\B(\H))$ are Banach algebras with the composition operation.
They have a natural multiplicative involution given by
$\Psi_1^\dagger(T) = \Psi_1(T^\ast)^\ast$ and satisfying that
$(\Psi_1 \Psi_2)^\dagger = \Psi_1^\dagger \Psi_2^\dagger$.

\begin{example}
  Recall that in
  the case of $\M = \ell_\infty(X) \subset \B(\ell_2 X)$ we have that
  \begin{eqnarray*}
    \CB^\sigma_{\ell_\infty(X) \, \ell_\infty(X)}(\B(\ell_2 X)) & = & \Schur^\sigma(X),\\
    \CB_{\ell_\infty(X) \, \ell_\infty(X)}(\B(\ell_2 X))        & = & \Schur(X).
  \end{eqnarray*}
  For non-discrete measure spaces $(X,\mu)$ we have that
  $\CB^\sigma_{L_\infty(X) \, L_\infty(X)}(\B(L_2 X))$ corresponds to the algebra
  of \emph{measurable Schur multipliers}, see \cite{Spronk2004}.
\end{example}

Now we are in position of stating the isomorphism between Haagerup tensors and
bimodular operators.

\begin{theorem}
  \label{S1.thm.HaagToCB}
  Let $\M \subset \B(\H)$ be a von Neumann algebra. The map $\Phi$ defined 
  by $x \otimes y \mapsto \Phi_{x \otimes y}$, where
  \[
    \Phi_{x \otimes y}(T) = x \, T \, y,
  \]
  extends to a surjective complete isometry and a $\dagger$-preserving
  homomorphism between the following spaces
  \begin{enumerate}[label=\emph{(\roman*)}, ref=(\roman*)]
    \item \label{S1.thm.HaagtoCB.1}
    $\Phi : \M \Haag \M \to \CB_{\M' \, \M'}(\K(\H))$.
    \item \label{S1.thm.HaagtoCB.2}
    $\Phi : \M \eHaag \M \to \CB^{\sigma}_{\M' \, \M'}(\B(H))$.
    \item \label{S1.thm.HaagtoCB.3}
    $\Phi : \M \nHaag \M \to \CB_{\M' \, \M'}(\B(H))$.
  \end{enumerate}
  Furthermore, the map in \emph{\ref{S1.thm.HaagtoCB.3}} is
  $\sigma(\M_\ast \eHaag \M_\ast)$ to $\sigma(\B(\H) \osprojtensor S_1(\H))$
  continuous and the map in \emph{\ref{S1.thm.HaagtoCB.2}} is both
  $\sigma(\M_\ast \Haag \M_\ast)$ to $\sigma(\K(\H) \osprojtensor S_1(\H))$
  continuous and $\sigma(\M_\ast \eHaag \M_\ast)$ to
  $\sigma(\B(\H) \osprojtensor S_1(\H))$ continuous.
\end{theorem}

The result above is well known to the experts, although their pieces are scattered
throughout the literature. We will just give a brief sketch with
references. Recall too that the first appearance of such result is credited
to be in an unpublished note of Haagerup \cite{HaagSD}. 

\begin{proof}
  Let us concentrate on \ref{S1.thm.HaagtoCB.2}, which will be the most
  important for our applications. The fact that $\Phi$ is a complete contraction
  amounts to a trivial calculation. Indeed, if $s = \sum_{j} x_j \otimes y_j$
  we may define, for every $1 \leq n$, the matrices
  \[
    \begin{array}{rc>{\displaystyle}lrc>{\displaystyle}l}
      x & = & \sum_{i = 0}^n \sum_{j} e_{i \, j} \otimes x_{j}, &
      y & = & \sum_{j = 0}^n \sum_{i} e_{i \, j} \otimes x_{i}
    \end{array}
  \]
  inside $\B(\ell_2) \weaktensor \M$, where $\{e_{i \, j}\}$ is a system
  of matrix units. Then $(\Id_{M_n} \otimes \Phi_s)(T)$ satisfies that
  \[
    (\Id_{M_n} \otimes \Phi_s)(T) = P_n \, x \, (\1 \otimes T) \, y \, P_n,
  \]
  where $P_n$ is the orthogonal projection on the span of
  $\{e_{j}\}_{j \leq n}$. Clearly
  \[
    \| \Id_{M_n} \otimes \Phi_s \| \leq \| x^\ast x \|^\frac{1}{2}_{\M} \, \| y \, y^\ast \|^\frac{1}{2}_{\M}
  \]
  and $\Phi$ is an $\M'$-bimodular operator. Taking the supremum over
  $n \geq 1$ and the infimum over all representations of $s$ gives that
  $\| \Phi_s \|_{\cb} \leq \| s \|_{\M \eHaag \M}$. To see that it is
  surjective notice that if $\Psi \in \CB_{\M' \, \M'}^\sigma(\B(\H))
  = \CB_{\M' \, \M'}(\K(\H),\B(\H))$ by Wittstock's factorization theorem
  for c.b. maps, see \cite{Pa1986Book}, we have that there
  is a large enough $\ell_2$ (we can take the dimension  of $\ell_2$ to be
  equal to that of $\H$ for infinite dimensional spaces), a representation
  $\pi : \K(\H) \to \B(\ell_2 \otimes_2 H)$ and two elements
  $x \in \B(\ell_2 \otimes_2 \H, \H)$, $y \in \B(\H, \ell_2 \otimes_2 \H)$
  such that $\Psi(x) = x \, \pi(x) \, y$ and
  $\| \Psi \|_{\cb} = \| x \| \, \| y \|$ but we can identify $x$ and
  $y$ with a row and a column respectively inside
  $\B(\ell_2) \weaktensor \B(\H)$ and  we have that $\Psi = \Phi_s$, where
  $s = x \odot y \in \B(\H) \eHaag \B(\H)$. It only rest to prove that if
  $\Psi$ is $\M'$-bimodular we can pick $x,y \in \B(\ell_2) \weaktensor \M$,
  which is the main result in \cite[Theorem 3.1]{Smith1991}.
  The rest of the points are similarly proved, see also
  \cite{BleSmi1992} for \ref{S1.thm.HaagtoCB.3}.
\end{proof}

As a consequence of the preceding theorem we are going to identify at times
$\M \eHaag \M$ and its weak-$\ast$ topology with
$\CB_{\M' \, \M'}^\sigma(\B(\H))$ and $\sigma(\B(\H) \osprojtensor S_1(\H))$.
The following lemma describe the weak-$\ast$ continuous functionals on
$\M \eHaag \M$ for its different preduals.

\begin{lemma}
  \label{S1.lem.dual}
  Let $\phi \in (\M \eHaag \M)^\ast$, then
  \begin{enumerate}[label=\emph{(\roman*)}, ref=(\roman*)]
    \item \label{S1.lem.dual.1} $\phi$ is
    $\sigma(\M_\ast \eHaag \M_\ast)$-continuous iff 
    \[
      \langle \phi, s \rangle = \big\langle C, (\Id \otimes \Phi_s)(B) \big\rangle,
    \] 
    where $C \in S_1(\ell_2 \otimes_2 \H)$ and $B \in \B(\ell_2 \otimes_2 \H)$
    \item \label{S1.lem.dual.2} $\phi$ is
    $\sigma(\M_\ast \Haag M_\ast)$-continuous iff
    \[
      \langle \phi, s \rangle = \big\langle C, (\Id \otimes \Phi_s)(B) \big\rangle,
    \]
    where $C \in S_1(\ell_2 \otimes_2 \H)$ and $B \in \K(\ell_2 \otimes_2 \H)$
  \end{enumerate} 
  Furthermore, $\phi$ is pointwise weak-$\ast$ continuous, iff
  $B$ in \emph{\ref{S1.lem.dual.1}} can be taken in
  $\B(\ell_2^n \otimes_2 H)$. Similarly, $\phi$
  is $\K(\H)$-pointwise weak-$\ast$ continuous iff we can take
  $B \in \K(\ell^n_2 \otimes_2 \H)$.
\end{lemma}

\begin{proof}
  We will prove \ref{S1.lem.dual.1} first. Since, by  Theorem
  \ref{S1.thm.HaagToCB}, the predual for the $\sigma(\M_\ast \Haag \M_\ast)$
  topology is given by
  $(\M \eHaag \M)_\ast = (\K(\H) \osprojtensor S_1(\H)) / F$, where $F$
  is the preannihilator of the $\M'$-bimodular maps, $\phi$
  can be lifted to an element (that we will denote also by $\phi$) in
  $\K(\H) \osprojtensor S_1(H)$ inducing the same functional. By definition of
  the o.s. projective tensor product we have that there are, possibly infinite,
  index sets $I_1$, $I_2$ and elements $A \in \K_{I_1} \mintensor S_1(H)$,
  $B \in \K_{I_2} \mintensor \K(\H)$ and 
  $\alpha, \beta \in S_2(\ell_2^{I_1},\ell_2^{I_2})$, where
  $\K_{I_i} = \K(\ell_2^{I_i})$, such that
  \[  
    \phi = \sum_{i,j \in I_1 \, p,q \in I_2} \alpha_{i p} (B_{i j} \otimes A_{p q}) \beta_{j q}.
  \]
  The action on $s \in \M \eHaag \M$ is given by
  \begin{eqnarray*}
    \langle \phi, s \rangle & = & \sum_{i,j \in I_1 \, p,q \in I_2} \alpha_{i p} \, \langle A_{i j}, \Phi_s(B_{p q}) \rangle \, \beta_{j q}\\
                            & = & \sum_{p,q \in I_2} \Big\langle \sum_{i, j \in I_1} \bar{\alpha}_{i p} \, A_{i j} \, \beta_{j q}, \Phi_s(B_{p q}) \Big\rangle\\
                            & = & \big\langle (\alpha^* \otimes \1) \, A \, (\beta \otimes \1), (\Id_{\K_{I_2}} \otimes \Phi_s)(B) \big\rangle.
  \end{eqnarray*}
  Note that, by \cite[Theorem 1.5]{Pi1998},
  $C = (\alpha^* \otimes \1) \, A \, (\beta \otimes \1) \in
  S_1(\ell_2^{I_2})[S_1(H)] \simeq S_1(\ell_2^{I_2} \otimes_2 H)$. We have
  thus that every weak-$\ast$ continuous functional $\phi$ can be expressed
  as
  \[
    \langle \phi, s \rangle = \big\langle C, (\Id_{\K} \otimes \Phi_s)(B) \big\rangle,
  \]
  concluding the proof of \ref{S1.lem.dual.1}. The same techniques yield
  \ref{S1.lem.dual.2}.
  
  The other claims in the statement follows by a repetition of the ideas used
  to prove that SOT-continuous and WOT-continuous functionals coincide over
  $\B(\H)$. Indeed, assume $\phi$ is pointwise weak-$\ast$ continuous. Then,
  there are finite collection $T_1, ..., T_m \in \B(\H)$ and
  $\xi_1, \xi_2, ..., \xi_m \in S_1(\H)$ such that $|\phi(\Psi)| < 1$
  whenever $|\langle \xi_i , \Psi(T_i)\rangle| < \epsilon$ for
  $i \in \{1,2,...,m\}$. In particular, taking
  $\Psi' = \Psi / \max\{ |\langle \xi_i , \Psi(T_i)\rangle| \}$ gives
  \[
    |\phi(\Psi)| \leq \epsilon^{-1} \, \max\{ |\langle \xi_i , \Psi(T_i)\rangle| \} \leq \epsilon^{-1} \, \sum_{i = 1}^m | \langle \xi_i, \Psi(T_i) \rangle |.
  \]
  As a consequence, if $\Psi(T_i) = 0$, for $i \in \{1,2,...,m\}$, we
  have $\phi(\Psi) = 0$ and so $\phi$ factors through a finite
  dimensional space. Therefore, $\phi$ can be expressed as a
  finite combination of simple tensors. 
\end{proof}

\section{{\bf The Correspondence Between Ideals and Modules}}

In this section we are going to prove the correspondence between left ideals in
$\M \eHaag \M$ and quantum relations over $\M$. We are going to start recalling
two easy lemmas that will be thoroughly used in this section. The
first asserts that the bilinear form $\odot$ can be extended
from $M_n[\M]$ to $\B(\ell_2) \weaktensor \M$, where $\weaktensor$ is the
weak-$\ast$ closed spatial tensor or equivalently, since $\B(\ell_2)$ is a
von Neumann algebra, the Fubini tensor product. The second is a stability
result for weak-$\ast$ closed left ideals in $\M \eHaag \M$. In the forthcoming
text we are going to denote by $\B(\ell_2) \weaktensor (\M \eHaag \M)$ the
weak-$\ast$ closed tensor product, with respect to the
$\sigma(\B(\H) \osprojtensor S_1(\H))$ topology. Recall that, using the
following identifications
\begin{eqnarray*}
  \B(\ell_2) \weaktensor (\M \eHaag \M) & \cong & \B(\ell_2) \weaktensor \CB^\sigma_{\M' \, \M'}(\B(\H))\\
                                        & \cong & \CB^\sigma_{\M' \, \M'}(\B(\H),\B(\ell_2 \otimes_2 \H)) 
\end{eqnarray*}
and reasoning like in \eqref{S1.eq.CBPredual}, we have that the predual of
$\B(\ell_2) \weaktensor (\M \eHaag \M)$ can be expressed as a quotient of
$\B(\H) \osprojtensor S_1(\ell_2 \otimes_2 \H)$.

\begin{lemma}
  \label{S2.lem.cntodot}
  The bilinear map $\odot: \B(\ell_2) \weaktensor \M \times \B(\ell_2)
  \weaktensor \M \to \B(\ell_2) \weaktensor (\M \eHaag \M)$ is bounded and
  continuous over bounded sets if $\B(\ell_2) \weaktensor \M \times \B(\ell_2)
  \weaktensor \M$ is given the product strong operator topology $\mathrm{(SOT)}$
  and $\B(\ell_2) \weaktensor (\M \eHaag \M)$ the
  $\sigma(\B(\H) \osprojtensor S_1(\H))$ topology.
\end{lemma}

\begin{proof}
  Let $(y_\alpha)_\alpha,
  (x_\alpha)_\alpha \subset \Ball(\B(\ell_2) \weaktensor \M)$ be nets in
  the unit ball satisfying that $x_\alpha \to x$ and $y_\alpha \to y$ in
  the SOT. Since the SOT and $\sigma$-SOT topologies agree on
  bounded set we can assume that we have SOT convergence for any given
  representation of $\B(\ell_2) \weaktensor \M$ and in particular for its
  representation on the Hilbert-Schmidt operators $S_2(\ell_2 \otimes_2 \H)$.
  Again since the weak-$\ast$ topology and the pointwise weak-$\ast$ topology
  of $\CB^\sigma_{\M' \, \M'}(\B(\H),\B(\ell_2 \otimes_2 \H))$ agree on bounded
  sets it is enough to see that for any $S \in \B(\H)$ and
  $\xi \in S_1(\ell_2 \otimes_2 \H)$, $\langle (S \otimes \xi), x_\alpha
  \odot y_\alpha \rangle \to \langle (S \otimes \xi), x \odot y \rangle$. But
  using that $\langle (S \otimes \xi), x \odot y \rangle =
  \langle \xi, x \, (\1 \otimes S) \, y \rangle$ and expressing
  $\xi = \eta \, \zeta^\ast$, where $\eta, \zeta$ are Hilbert-Schmidt
  operators, gives $\langle (S \otimes \xi), x \odot y \rangle =
  \langle \eta, x \, (\1 \otimes S) \, y \, \zeta \rangle$, where the last
  paring is just the inner product of $S_2(\ell_2 \otimes_2 \H)$. Using the
  SOT-convergence of $x_\alpha$ and $y_\alpha$ gives
  \[
    \begin{array}{r>{\displaystyle}l}
            & |\langle \eta, x_\alpha \, (\1 \otimes S) \, y_\alpha - x \, (\1 \otimes S) \, y \, \zeta \rangle | \\
      \leq  & |\langle \eta, x_\alpha \, (\1 \otimes S) \, (y_\alpha - y) \, \zeta \rangle| + |\langle \eta, (x_\alpha - x) \, (\1 \otimes S) \, y \, \zeta \rangle|\\
      \leq  & \Big( \sup_\alpha \| (\1 \otimes S^\ast) \, x_\alpha^\ast \, \eta \| \Big) \, \| (y_\alpha - y) \, \zeta \| + \|\eta\| \, \| (x_\alpha - x) \, (\1 \otimes S) \, y \, \zeta \|\\
      \to   & 0,
    \end{array}
  \]
  and that concludes the proof.
  \end{proof}

\begin{lemma}
  \label{S2.lem.odot}
  Let $J \subset \M \eHaag \M$ be a
  $\sigma(\B(\H) \osprojtensor S_1(\H))$-closed left ideal, the
  following holds
  \begin{enumerate}[label=\emph{(\roman*)}, ref=(\roman*)]
    \item \label{S2.lem.odot.1} If $X, Y \in \B(\ell_2) \weaktensor \M$
    satisfy that $X \odot Y \in \B(\ell_2) \weaktensor J$ then
    $Z X \odot Y T \in \B(\ell_2) \weaktensor J$ for every
    $Z,T \in \B(\ell_2) \weaktensor \M$.
    \item \label{S2.lem.odot.2} $X \odot Y \in \B(\ell_2) \weaktensor J$ if
    and only if $[X^\ast] \odot [Y] \in \B(\ell_2) \weaktensor J$.
  \end{enumerate}
\end{lemma}

\begin{proof}
  Since $J \subset \M \eHaag \M$ is a
  $\sigma(\B(\H) \osprojtensor S_1(\H))$-closed left ideal,
  $J' = \B(\ell_2) \weaktensor J$ is also a $\sigma(\B(\H) \osprojtensor
  S_1(\ell_2 \otimes_2 \H))$-closed left ideal.
  Furthermore, it satisfies that $J' (\B(\ell_2) \otimes \1) = J'$. For
  \ref{S2.lem.odot.1} just notice that if $Z = A \otimes x$,
  $T = B \otimes y$ are simple tensors, then
  \[
    Z X \odot Y T = (A \otimes x \otimes y) \, (X \odot Y) \, (B \otimes \1) \in J'
  \]
  Now, approximating $T$ and $Z$ by bounded, SOT-convergent nets of sums of
  simple tensor and applying \ref{S2.lem.cntodot} we obtain
  \ref{S2.lem.odot.1}.
  
  For \ref{S2.lem.odot.2} notice that if $[X^*] \odot [Y] \in J'$ then,
  by \ref{S2.lem.odot.1}, $X \, [X^*] \odot [Y] \, Y = X \odot Y \in J'$. For
  the other implication we just use functional calculus. Indeed, if
  $X \odot Y \in J'$ then $X^* X \odot Y Y^\ast \in J'$. Let us denote by
  $P = X^* X$ and $Q = Y Y^\ast$ and let $p_n(r)$ be a family of polynomials
  converging pointwise and boundedly to $\chi_{[0,\infty)}(r)$. Then,
  since all of the powers $P^n \odot Q^n$ lie in $J'$ we have that
  $p_n(P) \odot p_n(Q) \in J'$. Since $p_n(P) \to \chi_{[0,\infty)}(P) =
  [X^*]$ and $p_n(Q) \to \chi_{[0,\infty)}(Q) = [Y]$ in the SOT, we
  obtain the claim.
\end{proof}

We can now prove the main theorem of the section.

\begin{theorem}
  \label{S2.thm.Corr}
  Let $\M \subset \B(\H)$ be a von Neumann algebra. The maps
  \begin{equation*}
    \xymatrix{
      \left\{
        \substack{
          \R \subset \Prj(\B(\ell_2) \weaktensor \M)^2 \, : \\
          \R \mbox{ is an i.q.r.}
        }
      \right\}
      \ar@/_2pc/[d]_{J_{\R}}
      \ar@<-1pc>@/_5pc/[dd]_{\V_\R}
      \\
      \left\{
        \substack{
          J \subset \M \eHaag \M \, : \\
          (\M \eHaag \M) \, J \subset J,\\
          \sigma(\B \osprojtensor S_1)\mbox{-closed}
        }
      \right\}
      \ar@/_2pc/[d]_{\V_{J}}
      \ar@/_2pc/[u]^{\R_J}
      \\
      \left\{
        \substack{
          \V \subset \B(\H) \, : \\
          \M' \, \V \, \M' \subset \V, \mbox{ weak-}\ast \mbox{closed}
        }
      \right\}
      \ar@/_2pc/[u]^{J_\V}
      \ar@<-1pc>@/_5pc/[uu]_{\R_\V}
    }
  \end{equation*}
  given by
  \begin{eqnarray*}
    J_{\V}     & = & \{ s \in \M \eHaag \M : \Phi_s(T) = 0, \forall T \in  \V \} \\ 
    \R_\V      & = & \{ (P,Q) \in \Prj(\B(\ell_2) \weaktensor \M)^2 : \exists T \in \V, P (T \otimes \1) Q \neq 0 \} \\
    \V_J       & = & \{ T \in \B(\H) : \Phi_s(T) = 0, \forall s \in J \} \\
    \R_J       & = & \{ (P,Q) \in \Prj(\B(\ell_2) \weaktensor \M)^2 : P \odot Q \not\in \B(\ell_2) \weaktensor J \}\\
    J_\R       & = & \overline{ \{ (\phi \otimes \Id)(X \odot Y) : ([X^*],[Y]) \in \R, \phi \in \B(\ell_2)_\ast \}^{\mathrm{w}\ast} }\\
    \V_\R      & = & \{ T \in \B(\ell_2) : P ( T \otimes \1 ) Q = 0, \forall (P,Q) \notin \R \},
  \end{eqnarray*}
  are well defined, bijective and inverse of each other, i.e.
  \begin{multicols}{3}
    \begin{enumerate}[label=\emph{(\roman*)}, ref={(\roman*)}]
      \item \label{S2.thm.Corr.1} $\V_{J_\V} = \V$
      \item \label{S2.thm.Corr.2} $J_{\V_J} = J$
      \item \label{S2.thm.Corr.3} $\R_{J_\R} = \R$
      \item \label{S2.thm.Corr.4} $J_{\R_J} = J$
      \item \label{S2.thm.Corr.5} $\R_{\V_\R} = \R$
      \item \label{S2.thm.Corr.6} $\V_{\R_\V} = \V$
    \end{enumerate}
  \end{multicols}
  Furthermore, the rest of the maps commute, giving
  \begin{multicols}{3}
    \begin{enumerate}[label=\emph{(\arabic*)}, ref={(\arabic*)}]
      \item \label{S2.thm.Corr.C1} $\R_{J_\V} = \R_\V$
      \item \label{S2.thm.Corr.C2} $J_{\V_\R} = J_\R$
      \item \label{S2.thm.Corr.C3} $\V_{\R_J} = \V_J$
      \item \label{S2.thm.Corr.C4} $J_{\R_\V} = J_\V$
      \item \label{S2.thm.Corr.C5} $\R_{\V_J} = \R_J$
      \item \label{S2.thm.Corr.C6} $\V_{J_\R} = \V_\R$.
    \end{enumerate}
  \end{multicols}
\end{theorem}

\begin{proof}
  The fact that $\R_\V$ and $\V_\R$ are intrinsic
  quantum relations and weak-$\ast$ closed $\M'$-bimodules is trivial. Points
  \ref{S2.thm.Corr.5} and \ref{S2.thm.Corr.6} are the content of
  \cite[Theorem 2.32]{Wea2012}. We shall prove only the rest of the points.
  
  {{\bf Proof of \ref{S2.thm.Corr.1}.}}
  $\V_J$ is a weak-$\ast$ closed $\M'$-bimodule since it is the
  intersection of $\{ T \in \B(\H) : \Phi_s(T) = 0\}$ for every $s \in J$ and
  each of these subspaces is weak-$\ast$ closed and $\M'$-bimodular. It is also
  clear that $\V \subset \V_{J_\V}$,
  we only need to prove the converse. Let $T \not\in \V$. Since
  $\V \subset \B(\H)$
  is weak-$\ast$ closed there is, by the Hahn-Banach Theorem, a weak-$\ast$
  continuous functional $\phi : \B(\H) \to \CC$ such that $\phi(S) = 0,
  \forall S \in \V$ but $\phi(T) \neq 0$. Any such functional is of the form
  $\phi( A ) = \langle \eta, (\1 \otimes A) \xi \rangle$, where
  $\eta, \xi \in \ell_2 \otimes_2 \H$. Since $\V$ is an $\M'$-bimodule we
  have that
  \[
    \langle (\1 \otimes x) \, \eta, (\1 \otimes S) \, (\1 \otimes y) \, \xi \rangle = 0,
  \]
  where $S \in \V$ and $x,y \in \M'$. Let $P$ and $Q$ be the orthogonal
  projections onto the subspaces of $\ell_2 \otimes_2 \H$ given by
  \[
    \begin{array}{ccc}
      H_1 = \overline{(\1 \otimes \M') \, \eta},
      & &
      H_2 = \overline{(\1 \otimes \M') \, \xi}.
    \end{array}
  \]
  These subspaces are $(\1 \otimes \M')$-invariant, therefore
  $P, Q \in (\1 \otimes \M)' = \B(\ell_2) \weaktensor \M$. Clearly we have
  $P (\1 \otimes \V) Q = \{0\}$ but $P (\1 \otimes T) Q \neq 0$. Let us
  write $P = [p_{i j}]_{i,j}$ and $Q = [q_{i j}]_{i j}$, where
  $p_{i j}, q_{i j} \in \M$. Notice that:
  \[
    P (\1 \otimes T) Q = \left[ \Phi_{r_{i j}} (T) \right]_{i j},
  \]
  where
  \[
    r_{i j} = \sum_{k = 1}^\infty p_{i k} \otimes q_{k j} \in \M \eHaag \M.
  \]
  Since $P (\1 \otimes T) Q \neq 0$ there are $i, j$ such that
  $\Phi_{r_{i j}} (T) \neq 0$ but $r_{i j} \in J_{\V}$ which implies that
  $T \not\in \V_{J_{\V}}$ and so $\V_{J_\V} \subset \V$, which concludes
  \ref{S2.thm.Corr.1}.
  
  {{\bf Proof of \ref{S2.thm.Corr.2}.}}
  First, let us see that $J_\V$ is
  $\sigma(\B(\H) \osprojtensor S_1(\H))$-closed. Observe that $\Phi_s(T) = 0$
  iff $\langle \xi, \Phi_s(T) \rangle = 0$ for every $\xi \in \B(H)_\ast$.
  Therefore
  \[
    \{ s \in \M \eHaag \M : \Phi_s(T) = 0 \} = \bigcap_{\xi \in S_1(\H)} \{ s \in \M \eHaag \M : \langle \xi, \Phi_s(T) \rangle = 0 \},
  \]
  and so, the left hand side is pointwise weak-$\ast$ closed. Since the
  $\sigma(\B(\H) \osprojtensor S_1(\H))$ topology is finer than the
  pointwise weak-$\ast$ topology of $\CB_{\M' \, \M'}^\sigma(\B(\H))$ we
  have that $\{ s \in J : \Phi_s {|}_{\V} = 0\}$ is a weak-$\ast$ closed
  subspace. The fact that it is a left ideal follows trivially from the
  definition.
  
  Let $J$ be a $\sigma(\B(\H) \osprojtensor S_1(\H))$-closed left ideal.
  Again, it is clear that $J \subset J_{\V_J}$ we only have to prove the other
  containment. That is equivalent to prove that for every $s_0 \not\in J$ there
  is $T \in \B(H)$ such that $\Phi_s(T) = 0$ for every $s \in J$ and
  $\Phi_{s_0} \neq 0$. By weak-$\ast$ closeness of $J$ and the Hahn-Banach
  theorem there is a weak-$\ast$ continuous functional
  $\phi \in (\M \eHaag \M)_\ast$ such that $\langle \phi, s \rangle = 0$ for
  every $s \in J$ but $\langle \phi, s_0 \rangle \neq 0$. By Lemma
  \ref{S1.lem.dual} we have that 
  \[
    \langle \phi, s \rangle = \big\langle C, (\Id \otimes \Phi_s)(B) \big\rangle,
  \]
  where $C \in S_1(\ell_2 \otimes_2 H)$ and $B \in \B(\ell_2 \otimes_2 H)$. We
  can decompose $C = C_1 \, C_2^*$ where $C_1,C_2 \in S_2(\ell_2 \otimes_2 H)$
  and so
  \[
    \langle \phi, s \rangle = \big\langle C_1, (\Id \otimes \Phi_s)(B) \, C_2 \big\rangle
  \] 
  where $(\Id \otimes \Phi_s)(B) \in \B(\ell_2 \otimes_2 H)$ is
  acting on $S_2(\ell_2 \otimes_2 H)$ by left multiplication and the duality
  pairing is that of $S_2$ with itself. We have that, for every $s \in J$,
  $\langle \phi, s \rangle = 0$, and so, since $J$ is an ideal,
  $\langle \phi, (x \otimes y) \, s \rangle = 0$. Therefore
  \begin{equation}
    \label{S2.thm.Corr.eq1}
    0 = \big\langle (\1 \otimes x^*) \, C_1 , (\Id \otimes \Phi_s)(B) \, (\1 \otimes y) \, C_2 \big\rangle.
  \end{equation}
  Let us define the closed subspaces
  $H_1, H_2 \subset S_2(\ell_2 \otimes_2 H)$ given by
  \[  
    \begin{array}{rcl}
      H_1 =  \overline{ (1 \otimes \M) \, C_1}, & & H_2 =  \overline{ (1 \otimes \M) \, C_2}
    \end{array}
  \]
  and let $P_i : S_2(\ell_2 \otimes_2 H) \to H_i$, for $i \in \{1,2\}$, be
  their orthogonal projections. We can identify isometrically
  $S_2(\ell_2 \otimes_2 H)
  \cong \ell_2 \otimes_2 H \otimes_2 \ell_2 \otimes_2 H$,
  such identification gives that $\B(S_2(\ell_2 \otimes_2 H))
  \cong \B(\ell_2) \weaktensor \B(H) \weaktensor \B(\ell_2) \weaktensor \B(H)$,
  where the first two components correspond to right multiplication and the
  other two correspond to left multiplication. Since $H_1$ and $H_1$ are
  $\CC \1 \otimes \CC \1 \otimes \CC \1 \otimes \M$-invariant
  the projections $P_1$, $P_2$ belong to to
  $(\CC \1 \otimes \CC \1 \otimes \CC \1 \otimes \M)' =
  \B(\ell_2) \weaktensor \B(H) \weaktensor \B(\ell_2) \weaktensor \M'$. Now,
  the identity \ref{S2.thm.Corr.eq1} implies that
  \[
    0 = P_1 \, (\Id \otimes \Phi_s)(B) \, P_2,
  \]
  where $(\Id \otimes \Phi_s)(B)$ is seen as an operator in
  $\B(\ell_2) \weaktensor \B(H) \otimes \CC \1 \otimes \CC \1$. If
  $s = \sum_{k} x_k \otimes y_k$ we have that
  \begin{eqnarray*}
      P_1 \, (\Id \otimes \Phi_s)(B) \, P_2 & = & \sum_{k} P_1 \, (\1 \otimes \1 \otimes \1 \otimes x_k ) \, B \, (\1 \otimes \1 \otimes \1 \otimes y_k ) \, P_2\\
                                            & = & \sum_{k} (\1 \otimes \1 \otimes \1 \otimes x_k ) \, ( P_1 \, B \, P_2 ) \, (\1 \otimes \1 \otimes \1 \otimes y_k )\\
                                            & = & (\Id_{\B(\ell_2 \otimes_2 H \otimes_2 \ell_2)} \otimes \Phi_s) (P_1 \, B \, P_2).
  \end{eqnarray*}
  Let $T_\xi \in \B(H)$ be the operator given by
  $(\xi \otimes \Id_{\B(H)})(P_1 \, B \, P_1) \in \B(H)$, where
  $\xi \in \B(\ell_2 \otimes_2 H \otimes_2 \ell_2)_\ast$. We have that
  $\Phi_s(T_\xi) = 0$ for every $s \in J$ since
  \begin{eqnarray*}
    \Phi_s(T_\xi) & = & (\xi \otimes \Phi_s)(P_1 \, B \, P_1)\\
                  & = & \langle \xi, (\Id_{\B(\ell_2 \otimes_2 H \otimes_2 \ell_2)} \otimes \Phi_s)(P_1 \, B \, P_2) \rangle\\
                  & = & 0.
  \end{eqnarray*}
  But there has to be a
  $\xi_0 \in \B(\ell_2 \otimes_2 H \otimes_2 \ell_2)_\ast$ such that
  $\Phi_{s_0}(T_{\xi_0}) \neq 0$, otherwise
  \[
    \langle \xi, P_1 ( \Id_{\B(\ell_2 \otimes_2 H \otimes_2 \ell_2)} \otimes \Phi_{s_0} ) (B) P_2 \rangle = 0,
  \]
  for every $\xi \in \B(\ell_2 \otimes_2 H \otimes_2 \ell_2)_\ast$ which
  implies that $P_1 ( \Id \otimes \Phi_{s_0} ) (B) P_2 = 0$ but that is
  impossible since $C_1$ and $C_2$ are in the ranges of $P_1$ and $P_2$
  respectively. The existence of such $T_{\xi_0}$ finishes the proof.
 
  {{\bf Proof of \ref{S2.thm.Corr.3}.}}
  We will start proving that $\R_J$ is an intrinsic quantum relation. First, we
  have to see that $\R_J$ is weak-$\ast$ open. Since $J$ is
  $\sigma(\B(\H) \osprojtensor S_1(\H))$-closed, so is
  $\B(\ell_2) \weaktensor J$. The complementary
  $(\B(\ell_2) \weaktensor J)^{\mathrm{c}}$ is weak-$\ast$ open and so is
  $\R_J$, since it is the reverse image of
  $(\B(\ell_2) \weaktensor J)^{\mathrm{c}}$
  under the function $\odot: \Prj(\B(\ell_2) \weaktensor \M) \times
  \Prj(\B(\ell_2) \weaktensor \M) \to \B(\ell_2) \weaktensor (\M \eHaag \M)$,
  which is weak-$\ast$ continuous by Lemma \ref{S2.lem.odot} (recall that over
  $\Prj(\B(\ell_2) \weaktensor \M)$ the SOT, WOT, $\sigma$-SOT and
  $\sigma$-WOT coincide). Second, we are
  going to prove the properties \ref{S1.def.IQR.1}, \ref{S1.def.IQR.2},
  \ref{S1.def.IQR.2} in Definition \ref{S1.def.IQR}. It is trivial that
  $(0,0) \not\in \R_{J}$. For \ref{S1.def.IQR.2} we have to prove that
  \[
    \forall \alpha, \beta \, , (P_\alpha,Q_\beta) \in \B(\ell_2) \weaktensor J \, \Longleftrightarrow \big( \bigvee_{\alpha} P_\alpha, \bigvee_{\beta} Q_\beta \big) \in \B(\ell_2) \weaktensor J.
  \]
  For the implication ($\Longrightarrow$) we use that if
  $(P_\alpha,Q_\beta) \in \B(\ell_2) \weaktensor J$ then
  \[
    \big( \sum_{\alpha} P_\alpha, \sum_{\beta} Q_\beta \big) \in \B(\ell_2) \weaktensor J,
  \]
  but using that, for any family of projections $(R_\gamma)_\gamma$ 
  \[
    \big[ \sum_{\gamma} R_\gamma \big] = \bigvee_{\gamma} R_\gamma
  \]
  and Lemma \ref{S2.lem.odot} \ref{S2.lem.odot.2} we obtain that
  \[
    \big( \bigvee_{\alpha} P_\alpha, \bigvee_{\beta} Q_\beta \big) = \big( \big[ \sum_{\alpha} P_\alpha \big], \big[ \sum_{\beta} Q_\beta \big] \big) \in \B(\ell_2) \weaktensor J.
  \]
  Proving ($\Longleftarrow$) is clearly equivalent to proving that
  $P \odot Q \in \B(\ell_2) \weaktensor J$ implies that $R \odot S \in
  \B(\ell_2) \weaktensor J$ for any projections $R \leq P$ and $S \leq Q$, but
  that follows trivially from Lemma \ref{S2.lem.odot} \ref{S2.lem.odot.1}.
  For point \ref{S1.def.IQR.4} we have that if $P \odot [BQ] \in \B(\ell_2)
  \weaktensor J$ then $P \odot BQ \in \B(\ell_2) \weaktensor J$ by
  Lemma \ref{S2.lem.odot}. Since $B \in \B(\ell_2) \otimes \CC \1$ we have that
  $P \odot BQ = PB \odot Q \in \B(\ell_2) \weaktensor J$, again by Lemma
  \ref{S2.lem.odot}, that implies that $[B^\ast P] \odot Q \in \B(\ell_2)
  \weaktensor J$. The other implication is proved similarly. 
 
  In order to prove the inclusion $\R_{J_\R} \subset \R$ start by
  noticing that:
  \begin{eqnarray*}
    \B(\ell_2) \weaktensor J_\R & = & \B(\ell_2) \weaktensor \overline{ \{ (\phi \otimes \Id)( X \odot Y ) : \phi \in \B(\ell_2)_\ast, ([X^*],[Y]) \not\in \R \}^{\wast}}\\
                                & = & \wstspan \{ X \odot Y : ([X^\ast], [Y]) \not\in \R \}.
  \end{eqnarray*}
  If we assume that
  $P \odot Q \not\in \wstspan \{ X \odot Y : ([X^\ast], [Y]) \not\in \R \}$
  then trivially we have that $(P,Q) \in \R$. For the other inclusion we shall
  use that, by \ref{S2.thm.Corr.5}, $(P,Q) \in \R$ iff $(P,Q) \in \R_{\V_\R}$
  which happens only when $P \, (\1 \otimes A) \, Q \neq 0$ for some
  $A \in \V_\R$. Since the complete isometry
  $\Id \otimes \Phi : \B(\ell_2) \weaktensor (\M \eHaag \M)
  \to \CB^{\sigma}_{\M' \, \M'}(\B(\H),\B(\ell_2 \otimes_2 \H))$
  satisfies that $(\Id \otimes \Phi)_{X \odot Y}(A) = X \, (\1 \otimes A) \, Y$
  we have that
  \begin{eqnarray*}
    \wstspan \{ X \odot Y : ([X^\ast], [Y]) \not\in \R \} & =       & \wstspan \{ X \odot Y : [X^*] \, (\1 \otimes \V_\R) \, [Y] = \{0\} \}\\
                                                          & =       & \wstspan \{ X \odot Y : (\Id \otimes \Phi)_{[X^*] \odot [Y]}{|}_{\V_R} = 0 \}\\
                                                          & \subset & \{ s : (\Id \otimes \Phi)_s{|}_{\V_\R} = 0 \}.
  \end{eqnarray*}
  But no pair $(P,Q) \in \R$ satisfies that $P \odot Q \in
  \{ X \odot Y : ([X^\ast], [Y]) \not\in \R \}$ since that will imply that
  $(\Id \otimes \Phi)_{P \odot Q}{|}_{\V_\R} = 0$ and that is a contradiction.
  
  {{\bf Proof of \ref{S2.thm.Corr.4}.}}
  Let us see that $J_\R$ is an ideal for every intrinsic quantum relation
  $\R$ . To see that it is a
  linear subspace fix $\phi_1, \phi_2 \in \B(\ell_2)_\ast$ and $([X_1^*], [Y_1]) \not\in \R$,
  $([X_2^*], [Y_2]) \not\in \R$. If $B_1 : \ell_2 \to \ell_2$
  and $B_2 : \ell_2 \to \ell_2$ are isometries
  whose ranges are orthogonal and complementary. We have that the operators
  \begin{eqnarray*}
    X & = & (B_1 \otimes \1) \, X_1 \, (B_1^* \otimes \1) + (B_2 \otimes \1) \, X_2 \, (B_1^* \otimes \1)\\
    Y & = & (B_1 \otimes \1) \, Y_1 \, (B_1^* \otimes \1) + (B_2 \otimes \1) \, Y_2 \, (B_1^* \otimes \1)
  \end{eqnarray*}
  satisfy $[X^*] = (B_1 \otimes \1) \, [X_1^*] \, (B_1^* \otimes \1) +
  (B_2^* \otimes \1) \, [X_2^*] \, (B_1^* \otimes \1)$, $[Y] =
  (B_1 \otimes \1) \, [Y_1] \, (B_1^* \otimes \1) +
  (B_2^* \otimes \1) \, [Y_2] \, (B_1^* \otimes \1)$  and therefore,
  by \cite[Lemma 2.29]{KuWea2012}, $([X^*],[Y]) \in \R$. Now, a trivial
  calculation gives that
  \[
    \begin{array}{rl}
        & (\phi_1 \otimes \Id)(X_1 \odot Y_1) + (\phi_2 \otimes \Id)(X_2 \odot Y_2)\\
      = & \big( (B_1 \, \phi_1 \, B_1^* + B_2 \, \phi_2 \, B_2^*) \otimes \Id \big)(X \odot Y),
    \end{array}
  \]
  where $B_i \, \phi_j \, B_i^*(x) = \phi_j(B_i^* \, x \, B_i)$. 
  The fact that $J_\R$ is closed by scalar
  multiplication is trivial. It is also
  $\sigma(\B(\H) \osprojtensor S_1(\H))$-closed by construction.
  It only rest to see that it is absorbent for the multiplication.
  It is enough to prove that $(z \otimes t) \, J_\R \subset J_\R$ for every
  $z, t \in \M$. We have that
  \[
    (z \otimes t) \, (\phi \otimes \Id) (X \odot Y) = (\phi \otimes \Id)((\1 \otimes z) \, X \odot Y \, (\1 \otimes t)).
  \]
  Now, using that $[Y \, (\1 \otimes t)] \leq [Y]$ and
  $[X^* \, (\1 \otimes z^*)] \leq [X^*]$ and applying point \ref{S1.def.IQR.2}
  in Definition \ref{S1.def.IQR} gives the desired result.
  
  The inclusion $J_{\R_J} \subset J$ is easy to prove. Recall that if 
  $s \in \B(\ell_2) \weaktensor J$ then $(\phi \otimes \Id)(s) \in J$. Using
  that together with Lemma \ref{S2.lem.odot}\ref{S2.lem.odot.2} gives
  \[
    \begin{array}{rl}
                & \overline{\{(\phi \otimes \Id)(X \odot Y) : ([X^*], [Y]) \not\in \R_J, \phi \in \B(\ell_2)_\ast \}^{\wast}}\\
      =         & \overline{\{(\phi \otimes \Id)(X \odot Y) : [X^*] \odot [Y] \in \B(\ell_2) \weaktensor J, \phi \in \B(\ell_2)_\ast \}^{\wast}}\\
      \subset   & J.\\
    \end{array}
  \]
  For the reciprocal inclusion $J \subset J_{\R_J}$ we need to see that if
  $s \in J$ then there are $\phi \in \B(\ell_2)_\ast$, 
  $X, Y \in \B(\ell_2) \weaktensor \M$ with $([X^\ast],[Y]) \not\in \R_J$ such
  that $s = (\phi \otimes \Id)(X \odot Y)$. Note that we can express
  \[
    s = \sum_{k = 0}^\infty x_k \otimes y_k = (\omega_{e_1,e_1} \otimes \Id)(X \odot Y),
  \]
  as
  \[
    \begin{array}{>{\displaystyle}l>{\displaystyle}c>{\displaystyle}r}
      X = \sum_{k = 0}^\infty x_k \otimes e_{1 \, k} \in \B(\ell_2) \weaktensor \M
      & \mathrm{and} &
      Y = \sum_{k = 0}^\infty y_k \otimes e_{k \, 1} \in \B(\ell_2) \weaktensor \M.
    \end{array}
  \]
  We only have to prove that $([X^*], [Y]) \not\in \R_J$, i.e. that
  $[X^*] \odot [Y] \in \B(\ell_2) \weaktensor J$. Again, by Lemma
  \ref{S2.lem.odot}\ref{S2.lem.odot.2}, we only have to see that
  $X \odot Y \in \B(\ell_2) \weaktensor J$, which is equivalent to
  see that for every $\phi \in \B(\ell_2)_\ast$,
  $(\phi \otimes \Id)(X \odot Y) \in J$. Notice that if $P$ is the
  projection on the $1$-dimensional subspace spanned by $e_1$, then
  $X \odot Y = (P \otimes \1) \, (X \odot Y) \, (P \otimes \1)$. Therefore
  $(\phi \otimes \Id)(X \odot Y) = (P \, \phi \, P \otimes \Id)(X \odot Y)
  = (\lambda \, \omega_{e_1, e_1} \otimes \Id)(X \odot Y) = \lambda \, s \in J$,
  for some $\lambda \in \CC$. That finishes the proof of
  \ref{S2.thm.Corr.4}.

  Since we have already proved \ref{S2.thm.Corr.1}-\ref{S2.thm.Corr.6} we have
  that \ref{S2.thm.Corr.C4}-\ref{S2.thm.Corr.C6} can be deduced from
  \ref{S2.thm.Corr.C1}-\ref{S2.thm.Corr.C3}. We will prove only those first three
  cases, which are easy after the previous results.
  
  {{\bf Proof of \ref{S2.thm.Corr.C1}.}}
  We have that
  \begin{equation*}
    \R_{J_{\V}} = \{ (P,Q) \in \Prj(\B(\ell_2) \weaktensor \M)^2 : P \odot Q \not\in \B(\ell_2) \weaktensor J_\V \}
  \end{equation*}
  and that
  \[
    \B(\ell_2) \weaktensor J_\V = \{ s \in \B(\ell_2) \weaktensor (\M \eHaag \M) : (\Id \otimes \Phi)_s {|}_\V = 0 \},
  \]
  therefore
  \[
    \begin{array}{rclcl}
      \R_{J_{\V}} & = & \{ (P,Q) \in \Prj(\B(\ell_2) \weaktensor \M)^2 : (\Id \otimes \Phi)_{P \odot Q} {|}_\V \neq 0 \} &   & \\
                  & = & \{ (P,Q) \in \Prj(\B(\ell_2) \weaktensor \M)^2 : P (\1 \otimes \V) Q \neq \{0\} \}               & = & \R_\V.
    \end{array}
  \]
  
  {{\bf Proof of \ref{S2.thm.Corr.C2}.}}
  Let us start by seeing that $J_\R \subset J_{\V_\R}$. If
  $z = (\phi \otimes \Id)(X \odot Y)$, with $([X^*], [Y]) \not\in \R$, then
  $(\Id \otimes \Phi)_{z}(T) = (\phi \otimes \Id)(X \, (\1 \otimes T) \, Y)
  = (\phi \otimes \Id)(X \, [X^*] \, (\1 \otimes T) \, [Y^*] \, Y)$. So if
  $T \in \V_\R$ then $(\Id \otimes \Phi)_{z}(T) = 0$. For the converse
  inclusion let $s \in {J_{\V_\R}}$ and express $s$ as
  $s = (\omega_{e_1, e_1} \otimes \Id)(X \odot Y)$ like in the proof of
  \ref{S2.thm.Corr.4}. We have that $X \, (\1 \otimes \V_\R) \, Y  = 0$ and so
  $[X^*] \, (\1 \otimes \V_\R) \, [Y]  = 0$ which implies that
  $([X^*],[Y]) \not\in \R_{\V_\R} = \R$ and so $s \in J_\R$.
  
  {{\bf Proof of \ref{S2.thm.Corr.C3}.}}
  The inclusion $\V_J \subset \V_{\R_J}$ is trivial. In order to prove the
  converse, $\V_{\R_J} \subset \V_J$, fix $S \in \V_{\R_J}$. We have that
  $X \, (\1 \otimes S) \, Y = 0$, $\forall \, ([X^*], [Y])$ such that
  $[X^*] \odot [Y] \in \B(\ell_2) \weaktensor J$. Then, for any
  $\phi \in \B(\ell_2)_\ast$ we have that
  \[
    0 = (\phi \otimes \Id)(X \, (\1 \otimes S) \, Y)
      = (\phi \otimes \Id)((\Id \otimes \Phi)_{X \odot Y}(S))
      = \Phi_{(\phi \otimes \Id)(X \odot Y)} (S).
  \]
  Therefore $\Phi_{z}(S) = 0$ for every $z \in J_{\R_J} = J$.
\end{proof}

Recall that the technique of the proof of point \ref{S2.thm.Corr.1} follows
exactly the same lines of \cite[Lemma 2.8]{Wea2012}.

\begin{remark}
  Observe that, a priori, it is not clear why all
  $\sigma(\B(\H) \osprojtensor S_1(\H))$-closed ideals are closed in the
  coarser pointwise weak-$\ast$ topology. Such result is obtained as a
  consequence from Theorem \ref{S2.thm.Corr}.\ref{S2.thm.Corr.2}.
\end{remark}

\section{{\bf Invariant Quantum Relations \label{S5}}}

Let $\A$ be a von Neumann algebraic \emph{quantum group} with
comultiplication $\Delta$, see \cite{Daele2014} for a precise definition,
we will say that $\M$ is a \emph{quantum homogeneous space} if there
is a normal, $\ast$-homomorphism $\sigma: \M \to \A \weaktensor \M$, called
the \emph{coaction}, satisfying the natural coassociativity identity
\[
  (\Id \otimes \sigma) \, \sigma = (\Delta \otimes \Id) \, \sigma
\]
If $\M \subset \B(\H)$ is an \emph{standard form} for the von Neumann
algebra $\M$, we have that, after endowing $\H$ with its row (resp. column)
operator space structure, the coaction $\sigma$ extends to a complete
isometry $\sigma_2 : \H^\row \to \A \weaktensor H^\row$. We will say that an
operator $T \in \B(\H)$ is \emph{$\sigma$-equivariant} iff
\[
  \sigma_2 \, T = (\Id \otimes T) \, \sigma_2
\]
and we will denote by $\B(\H)^\sigma$ the space of $\sigma$-equivariat
operators. Similarly, we say that a quantum relation $\V$ over $\M$ is
\emph{$\sigma$-invariant}, or simply \emph{invariant} if the coaction
is understood from the context, iff
it is generated (as an operator $\M'$-bimodule) by $\sigma$-equivariant
operators. If $\V$ is generated by equivariant operators, then it is
generated by the equivariant operators inside $\V$, therefore $\V$ is
invariant iff 
\[
  \V = \leftidx{_{\M'}}{\left\langle \V \cap \B(\H)^\sigma \right\rangle }_{\M'}
     = \wstspan \{ x \, T \, y : x, y \in \M', \quad T \in \V \cap \B(\H)^\sigma \}.
\]
From now on we will denote $\V \cap \B(\H)^\sigma$ by $\V^\sigma$. Our purpose in
this section is to study invariant quantum relations. Interesting examples of
quantum homogeneous spaces include, among others, the ones listed bellow.

\begin{description}[leftmargin=0cm]
  \item[Classical homogeneous spaces] \label{S5.des.Cls} Let $G$ is a locally
  compact Hausdorff group and $X$ be a measurable $G$-space. $L_\infty(G)$
  is clearly a quantum group with the comultiplication given by 
  $\Delta(f)(g,h) = f(g \, h)$. Similarly, we can define the coaction
  $\sigma : L_\infty(X) \to L_\infty(G) \weaktensor  L_\infty(X)$ given by
  $\sigma(f)(g,x) = f(g \, x)$. To solidify our intuition let us see what
  happens when $X$ is discrete. In that case quantum relations over
  $L_\infty(X)$ are just subsets $R \subset X \times X$. Recall that a
  classical relation $R \subset X \times X$ is $G$-invariant iff
  \begin{equation}
    \label{S5.eq.defEq}
    (x,y) \in R \Longleftrightarrow (g \, x, g \, y) \in R, \mbox{  } \forall g \in G.
  \end{equation}
  We are going to see that such relations correspond with
  $\sigma$-invariant quantum relations. An operator 
  $T = [a_{x \, y}]_{x,y \in X} \in \B(L_2 X)$ is $\sigma$-equivariant iff it
  commutes with the action $\sigma_g(f)(x) = f(g^{-1} \, x)$, therefore the set
  \[
    R_T = \{ (x,y) \in X \times X : \langle e_x, T e_y \rangle \neq 0\} \subset X \times X
  \]
  satisfies \eqref{S5.eq.defEq} and the same goes for $R_\V$, where
  $\V = _{\R G}\langle \V^\sigma \rangle_{\R G}$, since
  \[
    R_\V = \bigcup_{T \in \V^\sigma} \{ (x,y) \in X \times X : \langle e_x, T e_y \rangle \neq 0 \}.
  \]
  This proves that any $\sigma$-invariant quantum relation over a discrete
  space $X$ corresponds to an invariant relation $R \subset X \times X$. The
  reciprocal is shown similarly.
  
  \item[Group von Neumann algebras] \label{S5.des.GvN} Let $G$ be a locally
  compact Hausdorff group, $L_2(G)$ the $L_2$-space with respect to the
  left Haar measure and $\lambda: G \to \mathcal{U}(L_2 G)$ be the unitary
  representation given by $\lambda_g(\xi)(h) = \xi(g^{-1} \, h)$, where
  $\xi \in L_2(G)$. The (left) group von Neumann algebra $\L G$ is given by
  \[
    \L G = \{ \lambda_g : g \in G \}'' \subset \B(L_2 G).
  \]
  The natural comultiplication structure $\Delta: \L G \to
  \L G \eHaag \L G$ is given by $\lambda_g \mapsto
  \lambda_g \otimes \lambda_g$. In this case the commutant $\L G'$ is given
  by the (right) group von Neumann algebra 
  \[
    \R G = \{ \rho_g : g \in G \}'' \subset \B(L_2 G)
  \]
  where $\rho_g(\xi)(h) = \xi(h \, g) \, \Delta(g)^{\frac{1}{2}}$, where
  $\xi \in L_2(G)$ is right regular
  representation. We can considere $\L G$ a quantum homogeneous space over
  itself with the multiplication as coaction. The representation
  $\L G \subset \B(L_2 G)$ is standard, and the $\Delta$-equivariant operators
  are given by the subalgebra $L_\infty(G)$ acting by multiplication operation.
  It is also illustrative to observe that if we take the GNS representation
  associated with the canonical Plancherel weight $\varphi$, see \cite{Ped1979},
  $\L G \subset \B(L_2(\L G, \varphi))$, then an element
  $T: L_2(\L G,\varphi) \to L_2(\L G, \varphi)$ is $\Delta$-equivariant iff it
  is a noncommutative Fourier multiplier over $L_2(\L G)$, in the sense of
  \cite[3.7]{CasSall2015}. By the Plancherel theorem, the algebra of such
  multipliers is equivalent to $L_\infty(G)$.
  
  \item[Quantum Torii] \label{S5.des.QTo} One family of von Neumann algebras
  that has received a considerable amount of attention is that of 
  \emph{quantum torii} $\A_\theta^n \subset L_2(\TT^n)$. In such case the
  coaction is given by $\sigma: \A_\theta^n \to
  L_\infty(\TT^n) \weaktensor \A_\theta^n$. Quantum
  relations on quantum torii have been considered before in
  \cite[Section 2.7]{Wea2012}.
\end{description}

Here, we will mainly focus our attention on the case of $\M = \L G$. Our
purpose is to describe the ideals associated with invariant quantum
relations over $\L G$. For that, we need to recall some results on the
representation of completely bounded $\R G$-bimodular operators
preserving the $\Delta$-equivariant operators. Let $M G$ be the Banach
algebra of finite measures with the o.s.s. given by $C_0(G)^* = M G$. Apart
from the weak-$\ast$ topology given by $\sigma(C_0 G)$ in $M G$ we have the
strictly finer topology generated by evaluation against every bounded
continuous function $\sigma(C_b G)$. Reasoning like before, since $M G$ is
$\sigma(C_b G)$-closed, the $\sigma(C_b G)$ topology induces another predual
for $M G$.
The subalgebra of point measures $\ell_1(G) \subset M G$ is of course
$\sigma(C_b G)$-dense. We define a multiplicative and injective map
$j : \ell_1(G) \to \L G \eHaag \L G$ by
$\delta_g \mapsto \lambda_g \otimes \lambda_{g^{-1}}$. The following theorem
assure that there is an injective and weak-$\ast$ continuous extension to $M G$
and characterizes its range as normal $\R G$-bimodular, c.b. maps preserving
$\B(L_2 G)^\Delta = L_\infty(G)$. We will denote the algebra of such operators
by $\CB_{\R G \mbox{-} \R G}^{\sigma, L_\infty(G)}(\B(L_2 G)) \subset
\CB_{\R G \mbox{-} \R G}^{\sigma}(\B(L_2 G))$. Such algebra is closed,
with respect to the natural weak-$\ast$ topologies of
$\CB_{\R G \mbox{-} \R G}^{\sigma}(\B(L_2 G))$ and so it inherits both the
$\sigma(\B(L_2 G) \osprojtensor S_1(L_2 G))$ and the
$\sigma(\K(L_2 G) \osprojtensor S_1(L_2 G))$ topologies. 

\begin{theorem}{{\bf (\emph{\cite[Theorem 3.2]{NeuRuanSpronk2008}})}}
  \label{S5.thm.NRS}
  Let $G$ be a locally compact group. There is a $\sigma(C_b)$ to
  $\sigma(\B \osprojtensor S_1)$ continuous,
  multiplicative and injective complete isometry
  $j: M G \to \L G \eHaag \L G$ extending the map
  $\delta_g \mapsto \lambda_g \otimes \lambda_{g^{-1}}$. Furthermore,
  the following diagram commute
  \begin{equation*}
    \label{S5.thm.NRS.dia}
    \xymatrix{
      M G \, \ar[d]^-{\Theta} \ar@{^{(}->}[rr]^j                                             & & \L G \eHaag \L G \ar[d]^-{\Phi}\\
      \CB_{\R G \mbox{-} \R G}^{\sigma, L_\infty(G)}(\B(\H)) \, \ar@{^{(}->}[rr]^{\subseteq} & & \CB_{\R G \mbox{-} \R G}^{\sigma}(\B(\H))
    }
  \end{equation*}
  
  In particular,
  $\Theta: MG \to \CB^{\sigma}_{\R G \mbox{-} \R G}(\B(L_2 G))$ is a
  complete isometry whose range is 
  $\CB^{\sigma, L_\infty(G)}_{\R G \mbox{-} \R G}(\B(L_2 G))$. The topology
  induced by $\sigma(\K \osprojtensor S_1)$ in $M G$ is the just
  $\sigma(C_0)$, while the topology induced by $\sigma(\B \osprojtensor S_1)$
  is $\sigma(C_b)$.
\end{theorem}

We will, perhaps ambiguously, denote by $\Theta$ either the map
$\Theta:M G \to \CB^{\sigma}_{\R G \mbox{-} \R G}(\B(L_2 G))$ or the
restriction to its image.

We will briefly sketch the proof of the theorem above since some of
its ideas will be used in the forthcoming results. But, before that, we
need to recall a few well known facts on the theory of crossed products.
Let $r: G \to \Aut(L_\infty G)$ be the normal right-translation action
given by $r_g(f)(x) = f(x \, g)$ noticing that by the Takai-Takesaki
duality theorem, see \cite{Tak1973Duality}, we have
\[
  L_\infty(G) \rtimes_{r} G = \B(L_2 G),
\]
where $\rtimes$ is notation for the (weak-$\ast$ closed) spatial crossed
product. The action $r$ is spatially implemented on
$L_\infty(G) \subset \B(L_2 G)$ by the right regular representation,
i.e. $r_g(f) = \rho_g \, f \, \rho_{g^{-1}}$ and so we obtain that
\[
  L_\infty(G) \rtimes_{r} G = \{L_\infty(G),\R G\}'' = \B(L_2 G).
\]
As a consequence, we can identify $L_\infty(G) \rtimes \1 \subset
L_\infty(G) \rtimes_r G = \B(L_2 G)$ with the algebra
of $\Delta$-equivariant operators.

\begin{proof}
  First, we are going to see that the map
  $\Theta: M G \to \CB_{\R G \mbox{-} \R G}^\sigma(\B(L_2 G))$ is surjective.
  Observe that $\Theta_\mu$ acts on $L_\infty(G) \subset \B(L_2 G)$ by left
  convolution, i.e. $\Theta_\mu(f) = \mu \ast f$. Notice also that, if
  $\Psi: \B(L_2 G) \to \B(L_2 G)$ is a normal and $\R G$-bimodular map,
  its restriction $\Psi{|}_{L_\infty(G)} : L_\infty(G) \to L_\infty(G)$
  determines the map $\Psi$ since $L_\infty(G)$ and $\R G$ generate the
  whole von Neumann algebra $\B(L_2 G)$ by the Takai-Takesaki theorem.
  Furthermore, since $\Psi$ preserves $L_\infty(G)$, we have that,
  for every $f \in L_\infty(G)$
  \[
    \rho_g \, \Psi(f) = \Psi(\rho_g \, f) = \Psi(\rho_g \, f \rho_{g^{-1}} \, \rho_g ) = \Psi(r_g \, f) \, \rho_g
  \]
  and so $\Psi{|}_{L_\infty (G)}$ is a right-translation equivariant
  operator, i.e. $r_g \, \Psi = \Psi \, r_g$. But then, any such operator is
  actually given by left convolution with a finite measure, see \cite{Wen1952}.
  So $\Psi(f) = \mu \ast f = \Theta_\mu(f)$ and since $\Psi$ and $\Theta$
  coincide in $L_\infty(G)$ they are equal.
  
  Reciprocally, if we pick a measure $\mu \in M G$ we have that the map
  $T_\mu: L_\infty(G) \to L_\infty(G)$ given by $f \mapsto \mu \ast f$
  is a normal bounded operator commuting with $r_g$ for all $g \in G$.
  Since $L_\infty(G)$ is an abelian operator space we have that $T_\mu$
  is c.b. and that
  \[
    \| T_\mu \|_\cb = \| \mu \|_{M G}.
  \]
  But for any crossed product there is a normal injective
  $\ast$-homomorphism $\iota : L_\infty(G) \rtimes G \to
  L_\infty(G) \weaktensor \B(L_2 G) \subset \B(L_2 G \otimes_2 L_2 G)$. To
  define such embedding $\iota$ fix an element $\xi \in L_2(G \times G)$. 
  The action of $\iota(f \, \rho_{g_0})$ on $\xi$ is given by
  \[
    j(f \, \rho_{g_0}) \, \xi (g,h) = f(g \, h^{-1} \, g_0^{-1}) \, \xi(g, g_{0}^{-1} \, h).
  \]
  while for general $x \in L_\infty(G) \rtimes_r G$ we just extend linearly
  and take weak-$\ast$ limits. Such embedding appears naturally in the crossed
  product construction. It satisfies that, if $T$ is equivariant, then
  \[
    \xymatrix{
      L_\infty(G) \rtimes G \, \ar@{^{(}->}[rr]^-{\iota} \ar[d]^-{T \rtimes \Id} & & L_\infty(G) \weaktensor \B(L_2 G) \ar[d]^-{T \otimes \Id}\\
      L_\infty(G) \rtimes G \, \ar@{^{(}->}[rr]^-{\iota}                         & & L_\infty(G) \weaktensor \B(L_2 G)    
    }.
  \]
  As a consequence, if $T$ is completely bounded so is $T \rtimes \Id$ and
  \[
    \| T \rtimes \Id \|_\cb = \| T \otimes \Id \|_\cb
  \]
  After identifying $L_\infty(G) \rtimes_r G$ with $\B(L_2 G)$, we get that
  $T_\mu \rtimes \Id$ is the only normal and $\R G$-bimodular extension of
  $T_\mu$. Therefore $\Theta_\mu = T_\mu \rtimes \Id$ and so $\Theta$ is
  well defined and isometric. Since $\Theta$ clearly factors through
  $\L G \eHaag \L G$ we also obtain that $j$ is a complete isometry.
\end{proof}

The result above goes back to Wendel \cite{Wen1952} but the formulation is
taken from \cite{NeuRuanSpronk2008}, whose main contribution is to generalize
the result from $\L G$ to its quantum group dual $L_\infty(G)$ obtaining a
complete isomorphism $\hat{\Theta}: M_\cb A G \to
\CB_{L_\infty(G) \, \mbox{-} L_\infty(G)}^{\sigma, \L G}(\B(L_2 G))$, where
$M_\cb A G$ is the space of completely bounded multipliers of the Fourier 
algebra $A G$. It is also worth pointing out that both results can be unified
using the language of quantum groups, see \cite{JunNeuRu2009}. 

The main result of this section is that invariant quantum relations over
$\L G$ are in bijective correspondence with weak-$\ast$ closed left
ideals inside $M G$.

\begin{theorem}
  \label{S5.thm.CorrInv}
  Let $\L G \subset \B(\H)$ be as above. If $\V$ is an invariant quantum
  relation over $\L G$ and $Q \subset M G$ is a $\sigma(C_b G)$-closed left
  ideal, then the following maps
  \begin{enumerate}[label=\emph{(\arabic*)}, ref={(\arabic*)}]
    \item $Q_\V = \{ \mu \in M G : \Theta_\mu{|}_\V = 0 \}$,
    \item $\V_Q = \{ T \in \B(\H) : \Theta_\mu(T) = 0, \quad \forall \mu \in Q \}$,
  \end{enumerate}
  are bijective and inverse of each other, see diagram below.
  \begin{equation*}
    \xymatrix{
      \left\{
        \substack{
          \V \subset \B(\H) \, : \\
          \V, \mbox{ invariant}
        }
      \right\}
      \ar@/_2pc/[rr]^-{Q_\V}
      &
      &
      \left\{
        \substack{
          Q \subset M G \, : (M G) \, Q \subset Q, \\
          \sigma(C_b G)\mbox{-closed}
        }
      \right\}
      \ar@/_2pc/[ll]_-{\V_Q}
    }
  \end{equation*}
\end{theorem}

Let us denote by $V_Q^\Delta$ the set
\[
  \V_Q^\Delta = \{ T \in \B(\H)^\Delta : \Theta_{\mu}(T) = 0, \quad \forall \mu \in Q\}.
\]
The proof of Theorem \ref{S5.thm.CorrInv} requires the following two lemmas.

\begin{lemma}
  \label{S5.lem.Fil}
  If $Q \subset MG$ is a $\sigma(C_b G)$-closed left ideal, then
  \begin{equation*}
    \label{S5.lem.Fil.1}
    \leftidx{_{\R G}}{\left\langle \V_Q^\Delta \right\rangle}_{\R G} = \V_Q.
  \end{equation*}
\end{lemma}

\begin{proof}
  After identifying $\B(L_2 G)$ with $L_\infty(G) \rtimes_r G$ again, we have
  that $\Theta_\mu = T_\mu \rtimes \Id$, where $T_\mu$ is the left convolution
  operator associated to $\mu$. We have that
   \[
     \begin{array}{rc>{\displaystyle}lrc>{\displaystyle}l}
       \V_Q & = & \bigcap_{\mu \in Q} \ker(T_\mu \rtimes \Id), & \V^\Delta_Q & = & \bigcap_{\mu \in Q} \ker(T_\mu).
    \end{array}
  \]
  But know we use that if $T$ is a $r$-equivariant operator then
  \[
    \begin{array}{rcl}
      \ker(T \rtimes \Id) & =       & \wstspan \left( \ker(T) \, \R G \right) \\
                          & =       & \wstspan \{ f \, \rho_g : f \in \ker(T), \, g \in G\}\\
                          & \subset & \leftidx{_{\R G}}{\left\langle \ker(T) \right\rangle }_{\R G}.
    \end{array}
  \]
  Using that the intersection of closures is larger that the closure of the
  intersections we get that
  \[
    \begin{array}{rr>{\displaystyle}lll}
      \V_Q & =       & \bigcap_{\mu \in Q} \ker(T_\mu \rtimes \Id)                                           &   &\\
           & \subset & \bigcap_{\mu \in Q} \leftidx{_{\R G}}{\left\langle \ker(T_\mu) \right\rangle }_{\R G} &   &\\
           & \subset & \leftidx{_{\R G}}{\Big\langle} \bigcap_{\mu \in Q} \ker(T_\mu) {\Big\rangle}_{\R G}   & = & \leftidx{_{\R G}}{\left\langle \V^\Delta_Q \right\rangle}_{\R G}.
    \end{array}
  \]
  The other inclusion is trivial since $\V_Q^\Delta \subset \V_Q$ and
  $\V_Q$ is a $\R G$-bimodule.
\end{proof}

\begin{lemma}
  \label{S5.lem.Corr}
  Let $Q \mapsto \V_Q^\Delta$ and $\V^\Delta \mapsto Q_{\V^\Delta}$ be as
  above, we have that
  \begin{enumerate}[label=\emph{(\arabic*)}, ref={(\arabic*)}]
    \item \label{S5.lem.Corr.1} $Q_{\V^\Delta_Q} = Q$.
    \item \label{S5.lem.Corr.2} ${\V^\Delta}_{Q_{\V^\Delta}} = \V^\Delta$.
  \end{enumerate}
\end{lemma}

\begin{proof}
  Let us start by \ref{S5.lem.Corr.1}. It is trivial that $Q \subset Q_{\V_Q}$.
  We only have to prove the reverse inclusion. Assume that $Q_{\V_Q}$
  is greater that $Q$. Then by the
  Hanh-Banach Theorem, for any $\mu_0 \in Q_{\V_Q} - Q$ we can take a
  functional $f_0 \in C_b(G)$ such that
  $\langle \mu_0, f_0 \rangle \neq 0$ but $\langle \mu, f_0 \rangle = 0$,
  for every $\mu \in Q$. Since $Q$ is a translation invariant space we
  have that $\mu \ast f_0 = 0$ for every $\mu \in Q$ but
  $\mu_0 \ast f_0 \neq 0$. The first condition implies that $f_0 \in \V_Q$,
  which contradict the fact that $\mu_0 \in Q_{\V_Q}$.
  
  For \ref{S5.lem.Corr.2} it is again clear that $\V^\Delta \subset \V_{Q_\V}$
  and we only have to prove the converse inclusion. By similar means using
  the Hanh-Banach theorem and the translation invariance of $V^\Delta$ we get
  the result.
\end{proof}

Now, we can proceed to prove the main correspondence theorem.

\begin{proof}{\bf (of Theorem \ref{S5.thm.CorrInv})}
  Let us start seeing that $Q_{\V_Q}$ is a $\sigma(C_b G)$-closed ideal. Notice
  that, $\mu \ast f = 0$ if and only if $\langle g, \mu \ast f \rangle = 0$
  for every $g \in L_1(G)$, but
  $\langle g, \mu \ast f \rangle = \langle \mu, \tilde{f} \ast g \rangle$,
  where $\tilde{f}(x) = f(x^{-1})$. Since $\tilde{f} \ast g$ is a right
  uniformly bounded function in $C_b(G)$, the kernel of
  $\mu \mapsto \mu \ast f$ is $\sigma(C_b G)$-closed and so
  is $Q_{\V_Q}$. The fact that $\V_Q$ is a weak-$\ast$ closed $\R G$-bimodule
  is immediate since $\Theta_\mu$ is weak-$\ast$ continuous $\R G$-bimodular
  map. The fact that is $\Delta$-invariant follows from \ref{S5.lem.Fil}.
  To prove that $Q = Q_{\V_Q}$ we just apply the following lemmas.
  \[
    \begin{array}{rclc}
      Q & = & Q_{\V^\Delta_Q}                                                      & \mbox{(by Lemma \ref{S5.lem.Corr})}\\
        & = & Q_{\leftidx{_{\R G}}{\left\langle \V_Q^\Delta \right\rangle}_{\R G}} &\\
        & = & Q_{\V_Q}                                                             & \mbox{(by Lemma \ref{S5.lem.Fil})}.
    \end{array}
  \]
  Similarly, taking the $\R G$-bimodules generated by the left and the right
  hand side of \ref{S5.lem.Corr.2}, we get that
  \[
    \leftidx{_{\R G}}{\left\langle \V^\Delta \right\rangle}_{\R G} 
    = 
    \leftidx{_{\R G}}{\big\langle} \V^\Delta_{Q_{\V^\Delta}} {\big\rangle}_{\R G}
    =
    \V_{Q_{\V^\Delta}}
    =
    \V_{Q_{\leftidx{_{\R G}}{\langle} \V^\Delta {\rangle}_{\R G}}}.
  \]
  But, by $\Delta$-invariance, the leftmost element is $\V$ and the rightmost
  is $\V_{Q_\V}$ and we conclude.
\end{proof}

\begin{remark}
  We have exposed here the theory of invariant quantum relations for $\L G$.
  The same proof above works, after \cite{NeuRuanSpronk2008} and
  \cite{JunNeuRu2009}, for a general quantum group $(\A,\Delta)$ just
  by replacing left ideals in $M G$ by left ideals in $M_\cb A(\A)$.
\end{remark}

Recall that if $G = \ZZ^n$, or any other abelian discrete group, then
$\L G = L_\infty(\TT^n)$ and any ideal $Q$ of $M G = \ell_1(\ZZ^n)$
correspond to a closed subset $C_Q \subset \hat{G}$ and such correspondence
in injective. Nevertheless, not every ideal in $M G$ is $\sigma(C_b G)$-closed
and therefore not all closed subsets will appear in the image of the
correspondence. We have that, in the invariant case, any quantum relation
$\V$ over $L_\infty(\TT^n)$ is actually a topological relation given by
\[
  (\theta_1, \theta_2) \in R \, \Longleftrightarrow \, \theta_1 ^{-1} \, \theta_2 \in C,
\]
where $C \subset \TT^n$ is a closed set. 

\section{\bf Remarks on $L_p$-$L_q$ versions of Quantum Relations}
In the introduction of \cite{Wea2012} it is stated that the
natural, albeit naive, candidate for quantized relations over a von
Neumann algebra $\M$ are the projections on $\M \weaktensor \M_\op$, but
that such object do not have desirable properties. The question of which
properties are missed is left unanswered there. Our aim here is to give an
intuitive explanation on why there is no well-behaved composition operation
between projection in $\M \weaktensor \M_\op$. After that, we will see that
there is a larger family of \emph{generalized quantum relations} that
contains both quantum relations and projections in $\M \weaktensor \M_\op$
as particular cases.

Start recalling that if $\M$ is a von Neumann algebra and $\phi$ is a normal,
faithful and semifinite weight we can define the noncommutative $L_p$-spaces
$L_p(\M,\phi)$, or simply $L_p(\M)$ if $\phi$ is understood from the context,
see \cite{PiXu2003}. Such spaces generalize the classical $L_p$-spaces $L_p(X,\mu)$
when $\M = L_\infty(X)$ and there are isomorphic identifications of
$L_\infty(\M)$ with $\M$, of $L_2(\M)$ with the GNS construction of $\phi$ and
of $L_1(\M)$ with the predual $\M_\ast$. There is a canonical o.s.s. for these
spaces given by operator space interpolation between
\begin{eqnarray*}
  L_1(\M)      & = & \M_\ast^\op\\
  L_\infty(\M) & = & \M,
\end{eqnarray*}
where the operator space structure of $\M^\op_\ast$ is given by restriction
of that of $(\M^\op_\ast)^{\ast \ast} = \M_\op^\ast$. Apart from being
compatible with interpolation, such spaces satisfy that $L_p(\M)^\ast =
L_{p'}(\M_\op)$, see \cite{Pi1998} as well as the remarks on
\cite[pp. 138]{Pi2003}. It is also known that $\CB(L_2(\M)) = \B(L_2(\M))$. The
spaces $L_p(\M)$ can be turned into $\M$-bimodules. Indeed, let we have two
commuting c.b. representations $l_p: \M \to \CB(L_p(\M))$ and
$r_p: \M_\op \to \CB(L_p(\M))$ generalizing the commuting actions in the
GNS construction of $\M$ when $p=2$. The module structure of
noncommtutative-$L_p$ has been studied in \cite{JunSher2005}. Let us denote by
$S' \subset \CB(L_p(\M))$ the commutant of $S \subset \CB(L_p(\M))$, by
\cite[Theorem 1.5]{JunSher2005}, we have that
\begin{eqnarray*}
  l_p[\M]'     & = & r_p[\M_\op]\\
  r_p[\M_\op]' & = & l_p[\M].
\end{eqnarray*}
Let us denote by $\CB_{p,q}$ the operator space given by
$\CB(L_p(\M),L_q(\M))$. Such spaces have a natural predual given by
\[
  \begin{array}{rllc}
    \CB(L_p(\M), L_q(\M)) & = & \CB(L_p(\M), L_{q'}(\M_\op)^*)                & \\
                          & = & \CB(L_p(\M), \CC) \Fubtensor L_{q'}(\M_\op)^* &  \mbox{(by \cite[Th. 4.1]{Pi2003})} \\
                          & = & (L_p(\M)  \osprojtensor L_{q'}(\M_\op))^*.    &
  \end{array}
\]
There are natural left actions on $\CB_{p,q}$ by $r_p[\M_\op]$ and $l_p[\M]$
and right actions by $r_q[\M_\op]$, $l_q[\M]$. We say that a subspace
$\V \subset \CB_{p,q}$ is a \emph{$(p,q)$-quantum relation} over
$\M$ iff $\V$ is weak-$\ast$ closed, with respect to the predual
$L_p(\M) \osprojtensor L_{q'}(\M_\op)$, and a
$r_p[\M_\op]$-$r_q[\M_\op]$-bimodule. It is easily shown that such relations
are independent of $\phi$. We also have the following.

\begin{proposition}
  \label{FW.prp.pqQ}
  \hfill
  \begin{enumerate}[label=\emph{(\roman*)}, ref=(\roman*)]
    \item \label{FW.prp.pqQ.1} Quantum relations over $\M$ correspond
    to $(2,2)$-quantum relations.    
    \item \label{FW.prp.pqQ.2} Projections in $\M \weaktensor \M_\op$ are in
    bijective correspondence with $(1,\infty)$-quantum relations.
  \end{enumerate}
\end{proposition}

\begin{proof}
  The proof of \ref{FW.prp.pqQ.1} is immediate since $\CB_{2,2} = \B(L_2(\M))$,
  $l_2: \M \to \B(L_2(\M))$ is the GNS representation of $\M$ and $r_2[\M]$ is
  just the commutant $M'$ for that representation. 
  In order to prove \ref{FW.prp.pqQ.2} we need to use that the map
  $j: \M \algtensor \M \to \CB(\M_\ast,\M)$ given by linear extension of
  \[
    j(x \otimes y)(\xi) = \langle y, \xi \rangle \, x 
  \]
  extends to a weak-$\ast$ continuous complete isomorphism
  $j: \M \weaktensor \M \to \CB(\M_\ast,\M)$, see
  \cite[Theorem 2.5.2]{Pi2003}.
  we have that
  \[
    \CB_{1,\infty} = \CB(L_1(\M), \M) = \CB(\M_\ast^\op, \M) = \M \weaktensor \M_\op.
  \]
  But, if $T = j(x \otimes y)$, we have that $r_\infty(z) \, T \, r_1(t) =
  j(z \, x \otimes y \, t) = j( (z \otimes t) \, (x \otimes y))$ and so a
  subspace $\V \subset \CB_{1,\infty}$ is
  $r_\infty[\M_\op]$-$r_1[\M_\op]$-bimodular iff, after seeing $\V$ as a subspace
  of $\M \weaktensor \M_\op$, it is a left ideal. Since the map $j$ is an isomorphism
  for the weak-$\ast$ topology, $\V \subset \M \weaktensor \M_\op$ is also
  weak-$\ast$ closed. But any weak-$\ast$ closed left ideal is of the form
  $\V = (\M \weaktensor \M_\op) \, P$, where $P \in \Prj(\M \weaktensor \M_\op)$.
\end{proof}

\begin{remark}
  The result above explains intuitively why we cannot expect to define a
  well-behaved composition operation between projections
  $P,Q \in \Prj(\M \weaktensor \M_\op)$. That composition will be carried
  to the composition of operators in $\CB(L_1(\M),\M)$ but that cannot be
  done, in general, since $\M$ does not embeds canonically in $L_1(\M)$.
\end{remark}

It is natural to ask whether $(p,q)$-quantum relations are in correspondence with
left ideals $J \subset \CB^\sigma_{r_q[\M_\op] \mbox{-} r_p[\M_\op]}(\CB_{p,q})$
suitably closed in some weak topology. The following proposition asserts that
this is the case when $(p,q) = (1,\infty)$.

\begin{proposition}
  The map $\Phi^{1,\infty}(x \otimes y) = l_\infty(x) \, l_1(y)$ extends to a
  weakly continuous complete isomorphism
  \begin{equation*}
     \Phi^{1,\infty}: \M \weaktensor \M_\op \to \CB^\sigma_{r_\infty[\M_\op] \mbox{-} r_1[\M_\op]}(\CB_{1,\infty}).
  \end{equation*}
  Under such correspondence any weakly closed left ideal $J$ is of the
  form $J = P \, (\M \weaktensor \M_\op)$ and its associated bimodule $V_J$
  corresponds, under the bijection in \emph{\ref{FW.prp.pqQ.2}}, to
  $P^\perp \in \Prj(\M \weaktensor \M_\op)$.
\end{proposition}

To prove the theorem above just notice that if $\N$ is a
von Neumann algebra, normal right $\N$-modular maps $T:\N \to \N$ are given
by left multiplication. Then, by applying that result to
$\N = \M \weaktensor \M_\op$ and using proposition \ref{FW.prp.pqQ} we conclude.

The discussion above leaves two natural open problems.
\begin{problem}
  \label{FW.pro}
  \hfill
  \begin{enumerate}[label= \textbf{(P\arabic*.)},ref=\textbf{(P\arabic*.)}]
    \item \label{FW.pro.1} Determine whether the double annihilator relation
    in Theorem \ref{S2.thm.Corr} between modules and ideals
    holds in general for $(p,q)$-quantum relations.
    \item \label{FW.pro.2} define an operator space tensor product
    $\otimes_{p,q}$ such that the map
    $\Phi^{p,q} : \M \algtensor \M_\op \to
    \CB^\sigma_{r_q[\M_\op] \mbox{-} r_p[\M_\op]}(\CB_{p,q})$ extends
    as a complete isometry to $\M \otimes_{p,q} \M_\op$.
  \end{enumerate}
\end{problem}

\section{\bf $W^\ast$-metrics and c.b. Gaussian bounds}
The aim of this section is to explain the original
motivation that guided us into studying quantum relations. Such motivation was
the necessity on \cite{GonJunPar2015} and \cite{GonJunPar2016QE} of expressing
\emph{off-diagonal} bounds in the context of noncommutative metric spaces.
Recall that in the classical case an operator $T = [a_{x,y}]_{x,y \in X}$
affiliated to $\B(\ell_2 X)$ has off-diagonal bounds if certain norms of
\[
  [a_{x \, y} \, \chi_{\{(x,y) : d(x,y) > r\}}]_{x,y \in X}
\]
decay in terms of $r > 0$. Earlier definition of \emph{quantum metric spaces} in
the $C^\ast$-algebraic framework, see \cite{Rief2004Gromov},
\cite{Rief2004Compact}, do not provide a natural way of formulating such
notion. On the other hand the notion of $W^\ast$-metric introduced by
Kuperberg and Weaver in \cite{KuWea2012} seems particularly well suited to the
task since a $W^\ast$-metric is a noncommutative generalization of the bundle
of band matrices of width $r > 0$. 

The upbringing of the notion of quantum relation is tightly connected with
the concept of $W^\ast$-metric space introduced by Kuperberg and Weaver in
\cite{KuWea2012}. Let us recall briefly such definition.

\begin{definition}{\bf (\cite[Definition 2.1(a)/2.3]{KuWea2012})}
  \label{FW.def.Wme}
  A family of subspaces $\VV = (\V_r)_{r \geq 0}$ of $\B(\H)$ is a
  \emph{$W^\ast$-pseudometric} over $\M \subset \B(H)$ iff
  \begin{enumerate}[label={(\roman*)}, ref=(\roman*)]
    \item \label{FW.def.Wme.1} Each $\V_r$ is a quantum relation over $\M$.
    \item \label{FW.def.Wme.2} Each $\V_r$ is symmetric, i.e.
    $\V_r^\ast = \V_r$.
    \item \label{FW.def.Wme.3} $\V_r \cdot \V_s \subset \V_{r + s}$.
    \item \label{FW.def.Wme.4} $\displaystyle{\bigcap_{s > t} \V_s = \V_t}$.
  \end{enumerate}
  We say that $\VV$ is a \emph{$W^\ast$-metric} iff $\V_0 = \M$.
\end{definition}

Notice that, if $\M = \ell_\infty(X)$ is a discrete measure space, then
every $\V_r$ corresponds to a relation $R_r \subset X \times X$. Condition
\ref{FW.def.Wme.2} becomes usual symmetry for $R_r$. Defining a function
$d_\VV(x,y) = \inf \{ r : (x,y) \in R_r \}$ gives that \ref{FW.def.Wme.4}
is the triangular inequality and so $d_\VV$ is a classical (pseudo)metric.

Classically, a metric measure space is a triple $(X,\mu,d)$ where $\mu$ is a
measure and $d$ is a metric such that the Borel $\sigma$-algebra generated by
$d$ is composed of measurable sets. The noncommutative version of a
measure space is generally regarded as a pair $(\M, \tau)$, where $\M$ is
a von Neumann algebra and $\tau: \M_+ \to [0,\infty]$ normal, semifinite and
faithful trace (or more generally a weight). Using $W^\ast$-metrics in this
context gives a good noncommutative generalization of metric measure spaces.
There are other, earlier, notions of quantized metric spaces, see for instance
\cite{Rief2004Gromov}, but $W^\ast$-metrics have some advantages. One of
them is that they provide a more natural framework for studying both finite
speed of propagation and \emph{off-diagonal} bounds associated with a
\emph{Markovian semigroup} over $(\M, \tau)$. Recall some definitions.

\begin{definition}
  \label{FW.def.Sem}
  A semigroup $(S_t)_{t \geq 0}$ of normal operators
  $S_t : \M \to \M$ is said to be Markovian iff
  \begin{enumerate}[label={(\roman*)}, ref=(\roman*)]
    \item \label{FW.def.Sem.1} Each $S_t$ is unital and completely positive.
    \item \label{FW.def.Sem.2} The semigroup is symmetric, i.e.
    $\tau( (S_t x)^\ast \, y) = \tau(x^\ast \, S_t y)$.
    \item \label{FW.def.Sem.3} The map $t \mapsto S_t$ is pointwise
    weak-$\ast$.
  \end{enumerate}
  Observe that as a consequence of $S_t$ being unital and
  \ref{FW.def.Sem.2} we get that $\tau \, S_t = \tau$.
\end{definition}

The most classical example of such type of semigroup is given by the heat
semigroup on $\RR^n$. In such case $S_t = e^{-t \, (-\Delta)}$ and its
kernel $k_t$ satisfies Gaussian bounds of the form
\begin{equation}
  \label{FW.eq.CGB}
  \big\| k_t (x,y) \, \chi_{\{ (x,y) : d(x,y) > r \} } \big\|_{L_\infty(\RR^n \times \RR^n)} \lesssim_{(n)} \frac{e^{- \frac{r^2}{4 \, t}}}{\sqrt{t^n}}.
\end{equation}
Such bounds have been used in the noncommutative case in \cite{GonJunPar2015}
with an ad hoc approach for $\M = L G$. Notice that, if $J_{\V_r}$ is generated
by a projection $P_r \in \Prj(\M \weaktensor \M_op)$ and the semigroup
$S_t$ can be expressed as an integral operator by
\[
  S_t(x) = \tau \big\{ k_t \, (\1 \otimes x) \big\},
\]
for some $k_t$ affiliated to $\M \weaktensor \M_\op$, then the
off-diagonal restriction is just $k_t \, P_r$ and we can generalize
\eqref{FW.eq.CGB} by bounding such element. Since, in general, ideals
in $\M \eHaag \M$ are not principal, such projection doesn't exist.
Nevertheless, we can take $\Phi_s(S_t)$ for $s \in \Ball(J_{\V_r})$, the
unit ball of $J_{\V_r}$, obtaining noncommutative Gaussian bounds of the
form
\[
  \sup_{s \in \Ball(J_{\V_r})} \big\| \Phi_s(S_t) \big\|_{\CB(L_1(\M), \M)} \lesssim \frac{e^{- \beta \frac{r^2}{t}}}{\sqrt{t^n}}.
\]

Another Harmonic analysis concept that seems natural to formulate in the
context of $W^\ast$-metrics is finite speed of propagation for the wave equation.
Recall that if $S_t = e^{-t \, (-\Delta)}$ is the heat equation in $\RR^n$ its
associated wave equation is given by
\[
   \partial_{t \, t}^2 f_t + (-\Delta) f_t = 0.
\]
The solution of such equation have \emph{finite speed of propagation}, meaning
that if $f_t$ is a solution of the above equation and $\supp[f_0] = K$, after
time $t > 0$ the support of $f_t$ is contained in
\[
  B_t(K) = \bigcup_{x \in K} B_t(x).
\] 
Such condition can be defined trivially using $W^\ast$-metrics as follows.

\begin{definition}
  We say that a Markovian semigroup over $S_t = e^{t A}$ have finite speed
  of propagation (with respect to some $W^\ast$-metric $\VV$) iff
  \[
    \cos( t \sqrt{A} ) \in \V_t, \quad \forall \, t > 0.
  \]
\end{definition}

Observe that the definition makes perfect sense since, without loss of
generality we can assume $\V \subset \B(L_2(\M,\tau))$ and clearly
$\cos(t \sqrt{A})$ is bounded in $L_2$. The intuition behind is that
$x_t = \cos(t \sqrt{A}) \, x$ satisfies the equation
$\partial_{t \, t}^2 x_t + A x_t = 0$ with $x_0 = x$. 

Gaussian bounds and finite speed of propagation are equivalent after
assuming certain hypothesis, see \cite{Si1996} \cite{Si2004}. Generalizing
such results and exploring the connections with locality in the noncommutative
setting is the goal of a forthcoming article.

\noindent \textbf{Acknowledgements.} Originally, the proof of the main
correspondence, i.e. Theorem \ref{S2.thm.Corr}, was a linearisation of the
proof given by Weaver in \cite[Theorem 2.32]{Wea2012} which used the fact that
pointwise norm closed ideals are weak-$\ast$ closed. The author is indebted
to prof. Marius Junge for pointing out an error in the proof and for
subsequent fruitful mathematical discussions. We are also thankful to him
for pointing out the existence of literature around Theorem \ref{S5.thm.NRS}.

\bibliographystyle{alpha}
\bibliography{../bibliography/bibliography}

\end{document}